\documentclass[a4paper]{amsart}

\usepackage{amsmath, amsfonts, amsthm , amssymb, amscd}
\usepackage{graphicx}
\usepackage{pst-grad}
\usepackage{url}
\usepackage{pstricks}
\usepackage[utf8]{inputenc}
\usepackage{pstricks-add}
\usepackage{pgf}
\usepackage{tikz}
\usepackage{tikz-cd}
\usepackage[thinlines]{easytable}
\usepackage{bm}
\usepackage{caption}

 \newtheorem{thm}{Theorem}[section]
 \newtheorem{prop}[thm]{Proposition}
 
 \newtheorem{cor}[thm]{Corollary}
\theoremstyle{definition}
 \newtheorem{rem}[thm]{Remark}
 \newtheorem{definition}{Definition}[section]
\numberwithin{equation}{section}
\newtheorem{example}{\sc Example}

\newcommand\mj{\mbox{\bf 1}}
\newcommand{\ket}[1]{| {#1} \rangle} 

\newcommand\cmd{\mathbf{Cmd}}
\newcommand\ps{\raisebox{0.7 pt}{${\pmb{\scriptstyle +}}$}}
\newcommand\ms{\raisebox{0.7 pt}{${\pmb{\scriptstyle -}}$}}

\newcommand\cob{\mathbf{Cob}}
\newcommand\cobG{\mathbf{Cob}_\mathfrak{G}}

\newcommand\fdhilb{\mathbf{fdHilb}}

\def\d#1{{#1\kern-0.4em\char"16\kern-0.1em}}
\def\D#1{{\raise0.2ex\hbox{-}\kern-0.4em#1}}
\def \Dj{\mbox{\raise0.3ex\hbox{-}\kern-0.4em D}}
\definecolor{britishracinggreen}{rgb}{0.0, 0.26, 0.15}
\definecolor{qqffqq}{rgb}{0.,1.,0.}
\definecolor{qqqqff}{rgb}{0.,0.,1.}
\definecolor{ffqqqq}{rgb}{1.,0.,0.}
\definecolor{ffxfqq}{rgb}{1.,0.4980392156862745,0.}
\definecolor{ududff}{rgb}{0.30196078431372547,0.30196078431372547,1.}

\title{A diagrammatic calculus for categorical quantum protocols}

\author[\D{D}or\D{d}evi\' c]{Du\v san \D{D}or\D{d}evi\' c}
\address{\scriptsize{Faculty of Physics \\ Studentski trg 12\\ 11001 Belgrade, Serbia}}
\email{dusan.djordjevic@ff.bg.ac.rs}
\author[Petri\' c]{Zoran Petri\' c}
\address{\scriptsize{Mathematical Institute SANU\\ Knez Mihailova 36, p.f.\ 367\\ 11001 Belgrade, Serbia}}
\email{zpetric@mi.sanu.ac.rs}
\author[Zeki\' c]{Mladen Zeki\' c}
\address{\scriptsize{Mathematical Institute SANU\\ Knez Mihailova 36, p.f.\ 367\\ 11001 Belgrade, Serbia}}
\email{mzekic@mi.sanu.ac.rs}

\date{}

\begin{document}

\begin{abstract}
As shown by Abramsky and Coecke, quantum mechanics can be studied in terms of dagger compact closed categories with biproducts. Within this structure, many well-known quantum protocols can be described and their validity can be shown by establishing the commutativity of certain diagrams in that category. In this paper, we propose an explicit realisation of a category with enough structure to check the validity of a certain class of quantum protocols. In order to do this, we construct a category based on 1-dimensional cobordisms with attached elements of a certain group freely generated by a finite set. We use this category as a graphical language, and we show that it is dagger compact closed with biproducts. Then, relying on the coherence result for compact closed categories, proved by Kelly and Laplaza, we show the coherence result, which enables us to check the validity of quantum protocols just by drawing diagrams. In particular, we show the validity of quantum teleportation, entanglement swapping (as formulated in the work of Abramsky and Coecke) and superdense coding protocol.


\noindent {\small {\it Keywords$\,$}: dagger compact closed
category, biproducts, cobordism, entanglement, coherence}
\end{abstract}

\maketitle

\section{Introduction}\label{uvod}
This paper offers a mathematical result and its application in the field of quantum information. The goal is to provide a minimal graphical language sufficient for verification of categorical quantum protocols. This is achieved through a category $1\cobG^\oplus$ based on the category of 1-dimensional cobordisms. We will work out in some details a few of quantum protocols and apply a technique of their verification based on the category $1\cobG^\oplus$. This work contains both pure mathematical results, as well as results in the applied field of quantum information, which are related to both physics and computer science. Therefore, we will try to make our exposition sufficiently detailed, in order to make the results available to a larger group of scientists from different areas.

There are several diagrammatic calculi proposed earlier (e.g.\ \cite{CO05,S07}) appropriate for verification of categorical quantum protocols. Our intention is to make one within a frame of a category relevant for quantum mechanics---the category of cobordisms. This is achieved by relying on a result from \cite{KL80} just by relaxing (i.e.\ neglecting some components of) the arrows of a category constructed in that paper. However, some labels of connected components of the underlying manifolds remain in our calculus, but we have minimized their role. One way to get rid of all labels is to increase the dimension of cobordisms in question. We discuss this suggestion in Section~\ref{labels}. Note that we are not trying to provide a diagrammatic calculus that is better for practical purposes that the existing ones (for example, ZX calculus \cite{CD11}), but rather to, in a very precise manner, formulate one calculus that uses geometry and that is able to capture some aspects of quantum mechanics. We do not claim that we can formulate and check the validity of all quantum protocols (for example, it is hard to simulate the most general unitary operator acting on two qubits).

Ever since its formulation in the first half of the 20-th century, quantum mechanics is naturally set to live in a separable Hilbert space. This enables one to talk about notions as entanglement and measurement in an almost trivial way. However, they are far from being understood by a scientific community. While the basic mathematical formalism is easy to understand, its physical meaning is much less clear and various interpretations of quantum mechanics are possible, without currently any basis on which we could select only one correct. This means that it could be fruitful to reconsider some basic notion about quantum mechanics, and to try to formulate it in terms of a different mathematical structure. Similar motivation led authors of \cite{AC04} to develop the so called categorical approach to quantum mechanics in terms of dagger compact closed categories with biproducts.
Of course, the starting point is naturally the case where standard Hilbert space formalism leads to the consideration of finite dimensional spaces (for example, consideration of spin $1/2$ gives rise to the two dimensional Hilbert space; on the other hand, it is not too hard to construct classical theories whose Hilbert space after quantization turns out to be finite dimensional, \cite{T14}, although they are much less known for an average physicist). For simplicity, and for the practical application, we restrict to the case of two dimensional Hilbert spaces (qubits). Basis vectors are $\ket{0}$ and $\ket{1}$ (we use Dirac notation for this part). If we denote $\mathcal{H}={\rm span}\,(\ket{0},\ket{1})$, then the total Hilbert space of a composite bipartitive system is $\mathcal{H}\otimes \mathcal{H}$. As we use tensor product, there are no projections to $\mathcal{H}$, and we can introduce the notion of an entangled state. For example, we can form a basis in $\mathcal{H}\otimes \mathcal{H}$ from the entangled states as
\begin{align*}
&\ket{\beta_1}=\frac{1}{\sqrt{2}}(\ket{0}\otimes\ket{0}+\ket{1}\otimes\ket{1}),\\
&\ket{\beta_2}=\frac{1}{\sqrt{2}}(\ket{0}\otimes\ket{1}+\ket{1}\otimes\ket{0}),\\
&\ket{\beta_3}=\frac{1}{\sqrt{2}}(\ket{0}\otimes\ket{1}-\ket{1}\otimes\ket{0}),\\
&\ket{\beta_4}=\frac{1}{\sqrt{2}}(\ket{0}\otimes\ket{0}-\ket{1}\otimes\ket{1}).
\end{align*}
Basis $\{\beta_i\}_{i=1}^{i=4}$ is usually referred to as Bell basis (see e.g. \cite{NC10}). Of course, this construction is made on a few assumptions. The first one is a linear structure of vector spaces. Historically, linear structure was a natural guess based on an intuition about wavelike properties of particles (electron). For example, it was known long before the birth of quantum mechanics that light can be described in terms of oscillating electric and magnetic field, and that for those fields superposition principle holds (linearity of Maxwell's equations). Moreover, intensity of a wave is proportional to $|\mathbf{E}|^2$ (where $\mathbf{E}$ stands for a complex representative of electric field), and this further motivates the Born rule  for a probabilities associated with the measurement outcome. Despite its success, it is still intriguing to consider theories without vector space structure.

Quantum mechanics can be considered as a special case of quantum field theories for $0+1$ dimensions. On the other hand, it is well known that one approach to quantum field theories (especially to the case of topological quantum field theories) is using cobordisms to represent space-time evolution processes. This opens another natural question, and that is to what extent one can use the category $1\cob$ (of 1-dimensional cobordisms) to simulate quantum mechanical processes. In addition, there are ideas from quantum gravity, in the context of AdS/CFT correspondence, that some aspects of quantum theory (for example, entanglement entropy) could be obtained from geometry \cite{RT06}. Quantum-geometry relation is also evident in $\mathrm{ER}=\mathrm{EPR}$ proposal \cite{M13}. Therefore, seeking the role of geometry in quantum physics is interesting on its own.  The results of our paper can be viewed as a step toward understanding the connection between quantum mechanics and geometry.

In order to obtain this correspondence, we will introduce the notion of $\mathfrak{G}$-cobordisms. They correspond to a regular cobordisms (for a precise definition of 1-dimensional cobordisms see the following sections), but with additional structure, such that each connected component has an element of a group $\mathfrak{G}$ attached. This introduces a notion of a $\mathfrak{G}$-segment or a $\mathfrak{G}$-circle. Group elements will play the role of (unitary) transformations that can be done on a quantum state. Note that similar idea was discussed in \cite{K10}, but without explicit referring to the categorical quantum mechanics.

On the other hand, a motivation for our work can be purely mathematical. In category theory and its application, it is of great importance to establish whether a certain diagram commutes. Usually, this is done by inspection, using a set of equalities (for example as those from Appendix~\ref{language}). Though in principle a straightforward task, it usually consumes a non-negligible amount of time. For this reason, it is practical to prove certain coherence results. Such a result enables one to check the commutativity of diagrams, consisting of canonical arrows of a certain categorical structure, just by drawing pictures in an appropriate graphical language. A detailed explanation of our approach to coherence is given in \cite[Introduction]{PZ21}, where also results akin to those proven here are presented. Briefly, we start with a freely generated category built out of syntax material, whose objects are formulae and arrows are equivalence classes of terms in an equational system. Then we show a completeness result with respect to a model in a form of a graphical category.

In Section \ref{closed} we review some basic categorical notions relevant for this paper. In Section \ref{jedancob} we further discuss the category $1\cob$. Section \ref{free category} introduces two compact closed categories with some additional structure both freely generated by a free group considered as a category. The isomorphism of these two categories is established in that section. Next two sections (Sections \ref{injekcije} and \ref{normalizacija}) are technical necessity, and could be skipped in the first reading. In Section \ref{graficki} we introduce gradually (in several steps) the category $1\cobG^\oplus$, and formulate the coherence result essential for the verification of certain categorical quantum protocols, a task we perform in Section~\ref{validity}. A possibility of omitting the labels via increasing the dimension of cobordisms is given in Section~\ref{labels}. We give our final remarks in Section~\ref{conclusion}. Appendix A contains an equational presentation of dagger compact closed categories with dagger biproducts and Appendix B discusses a categorical approach to scalars and probability amplitudes.

\section{Closed categories and biproducts}\label{closed}

Some notions from category theory relevant for this paper are introduced in this section. A \emph{symmetric monoidal category} is a category $\mathcal{A}$ equipped with a distinguished object $I$, a bifunctor $\otimes\colon \mathcal{A}\times \mathcal{A}\to \mathcal{A}$ (we abbreviate $\mj_a\otimes f$ and $f\otimes\mj_a$ by $a\otimes f$ and $f\otimes a$, respectively) and the natural isomorphisms $\alpha$, $\lambda$ and $\sigma$ with components $\alpha_{a,b,c}\colon a\otimes(b\otimes c)\to (a\otimes b)\otimes c$, $\lambda_a\colon I\otimes a\to a$ and $\sigma_{a,b}\colon a\otimes b\to b\otimes a$. (Note that in \cite{AC04}, $\lambda$ denotes the inverse of our $\lambda$ and due to the presence of symmetry, we do not introduce a name for the isomorphism $a\cong a\otimes I$.) Moreover, the coherence conditions concerning the arrows of $\mathcal{A}$ (see the equalities~\ref{19}-\ref{21} in Appendix~\ref{language}) hold. A symmetric monoidal category is \emph{monoidally strict} when the operation $\otimes$ on its objects is associative with $I$ being the neutral, and moreover, the arrows $\alpha$ and $\lambda$ are identities.

A \emph{compact closed category} is a symmetric monoidal category in which every object $a$ has its \emph{dual} $a^\ast$. This means that there are \emph{units} $\eta_a\colon I\to a^\ast\otimes a$ and \emph{counits} $\varepsilon_a\colon a\otimes a^\ast\to I$ such that the equalities~\ref{23} of Appendix~\ref{language} hold. If a functor between two compact closed categories preserves this structure ``on the nose'', then we say that it \emph{strictly preserves} the compact closed structure, and we use the same terminology in other cases.

It is straightforward to conclude that the following isomorphisms hold in every compact closed category.
\[
u_{a,b}\colon (a\otimes b)^\ast\cong b^\ast\otimes a^\ast,\quad v\colon I^\ast\cong I,\quad w_a\colon a^{\ast\ast}\cong a
\]
(In a monoidally strict compact closed category
\[
u_{a,b}=(b^\ast\otimes a^\ast\otimes\varepsilon_{a\otimes b})\circ(b^\ast\otimes \eta_a\otimes b\otimes (a\otimes b)^\ast)\circ (\eta_b\otimes (a\otimes b)^\ast)
\]
$v=\varepsilon_I$ and $w_a=(\varepsilon_{a^\ast}\otimes a)\circ (\sigma_{a^{\ast\ast},a^\ast}\otimes a)\circ (a^{\ast\ast}\otimes\eta_a)$.) A compact closed category is \emph{strict} when it is monoidally strict and $(a\otimes b)^\ast= b^\ast\otimes a^\ast$, $I^\ast= I$ and $a^{\ast\ast}= a$, while $u_{a,b}$, $v$ and $w_a$ are identities.

For quantum protocols discussed below, the following derived operations on arrows of a compact closed category are frequently used. For $f\colon a\to b$ its \emph{name} $\ulcorner f \urcorner\colon I\to a^\ast\otimes b$ and its \emph{coname} $\llcorner f \lrcorner\colon a\otimes b^\ast\to I$ are defined as
\[
\ulcorner f \urcorner=(a^\ast\otimes f)\circ \eta_a,\quad \llcorner f \lrcorner=\varepsilon_b\circ(f\otimes b^\ast).
\]
The function $\ast$ on objects of a compact closed category $\mathcal{A}$, extends to a functor $^\ast\colon\mathcal{A}^{\mathrm{op}}\to \mathcal{A}$ in the following way. For $f\colon a\to b$, let $f^\ast\colon b^\ast\to a^\ast$ be
\[
\lambda_{a^\ast}\circ \sigma_{a^\ast,I}\circ (a^\ast\otimes \varepsilon_b) \circ \alpha^{-1}_{a^\ast,b,b^\ast}\circ((a^\ast\otimes f)\otimes b^\ast)\circ(\eta_a\otimes b^\ast)\circ\lambda_{b^\ast}^{-1}.
\]

A \emph{dagger category} is a category $\mathcal{A}$ equipped with a functor $\dagger\colon \mathcal{A}\to \mathcal{A}^{\mathrm{op}}$ such that for every object $a$ and every arrow $f$ of this category $a^\dagger=a$, and $f^{\dagger\dagger}=f$. (For more details see \cite{S07,H09}.) A \emph{dagger} compact closed category is a compact closed category $\mathcal{A}$, which is also a dagger category satisfying the equalities \ref{28}-\ref{30} of Appendix~\ref{language}.

By composing the functors $\dagger$ and $^\ast$ one obtains the functor $_\ast=\, ^\ast\circ\dagger\colon \mathcal{A}\to \mathcal{A}$ ($a_\ast=a^\ast$, $f_\ast=(f^\dagger)^\ast$). For a strict dagger compact closed category $\mathcal{A}$, the functor $_\ast$ satisfies
\[
f_{\ast\ast}=f,\quad (f_\ast)^\ast=(f^\ast)_\ast.
\]

A \emph{zero-object} is an object which is both initial and terminal. For a category with a zero-object $0$ there is a composite $0_{a,b}\colon a\to 0\to b$ for every pair $a$, $b$ of its objects, and for every other zero-object $0'$ of this category, the composite $a\to 0'\to b$ is equal to $0_{a,b}$. A  \emph{biproduct} of $a_1$ and $a_2$ in a category with a zero-object consists of a coproduct and a product diagram
\[
a_1\stackrel{\iota^1\:}{\longrightarrow} a_1\oplus a_2 \stackrel{\:\iota^2}{\longleftarrow}a_2,\quad\quad\quad
a_1\stackrel{\:\pi^1}{\longleftarrow} a_1\oplus a_2 \stackrel{\pi^2\:}{\longrightarrow}a_2
\]
for which
\begin{equation}\label{pijota}
\pi^j\circ\iota^i= \left\{\begin{array}{ll}
    \mj_{a_i}, & i=j,
    \\[1ex]
    0_{a_i,a_j}, & \mbox{\rm otherwise}, \end{array} \right .
\end{equation}
where $i,j\in\{1,2\}$ (cf.\ the equalities \ref{13}-\ref{14} in Appendix~\ref{language}). For arrows $f_1\colon a_1\to c$ and $f_2\colon a_2\to c$, the unique arrow $h\colon a_1\oplus a_2\to c$ for which $h\circ\iota^i=f_i$, $i\in\{1,2\}$ is denoted by $[f_1,f_2]$, and for arrows $g_1\colon c\to a_1$ and $g_2\colon c\to a_2$, the unique arrow $h\colon c\to a_1\oplus a_2$ for which $\pi^i\circ h=g_i$, $i\in\{1,2\}$ is denoted by $\langle f_1,f_2\rangle$.

More generally, a biproduct of a family of objects $\{a_j\mid j\in J\}$ consists of a universal cocone (coproduct diagram) and a universal cone (product diagram)
\[
\{\iota^j\colon a_j\to \oplus_{j\in J} a_j\mid j\in J\}, \quad\quad\quad \{\pi^j\colon \oplus_{j\in J} a_j\to a_j\mid j\in J\}
\]
for which the equality~\ref{pijota} holds for all $i,j\in J$. A \emph{category with biproducts} is a category with a zero-object and biproducts for every pair of objects. A biproduct is a \emph{dagger biproduct} when for every pair $a$, $b$ of objects the equalities~\ref{31} of Appendix~\ref{language} hold.

For $f,g\colon a\to b$ in a category with biproducts whose \emph{codiagonal} and \emph{diagonal} maps are $\mu_b\colon b\oplus b\to b$ and $\bar{\mu}_a\colon a\to a\oplus a$ one defines $f+g$ as $\mu_b\circ(f\oplus g)\circ\bar{\mu}_a$. This operation on the set ${\rm Hom}\,(a,b)$ of arrows from $a$ to $b$ is commutative and has $0_{a,b}$ as neutral. Moreover, the composition distributes over $+$. Hence, every category with biproducts may be conceived as a category enriched over the category $\cmd$ of \emph{commutative monoids}.

Alternatively, to define biproducts in a category enriched over $\cmd$ it suffices to assume the existence of a bifunctor $\oplus$, a special object $0$, and for every pair of objects $a$, $b$ the arrows $\pi^1_{a,b}\colon a\oplus b\to a$, $\pi^2_{a,b}\colon a\oplus b\to b$, $\iota^1_{a,b}\colon a\to a\oplus b$ and $\iota^2_{a,b}\colon b\to a\oplus b$, for which the equalities \ref{10}-\ref{22} of Appendix~\ref{language} hold. As a justification of this approach see the proof of Corollary~\ref{3.3} below.

In a compact closed category with biproducts, tensor distributes over $\oplus$, i.e. there exist \textit{distributivity isomorphisms} $\tau_{a,b,c}\colon a\otimes (b\oplus c)\to (a\otimes b)\oplus (a\otimes c)$ and $\upsilon_{a,b,c}\colon (a\oplus b)\otimes c\to (a\otimes c)\oplus (b\otimes c)$ explicitly given by
\begin{align}
\tau_{a,b,c}&=\langle\mj_a\otimes\pi^{1}_{b,c}, \mj_a\otimes\pi^{2}_{b,c}\rangle, \qquad \tau^{-1}_{a,b,c}=[\mj_a\otimes \iota^1_{b,c}, \mj_a\otimes\iota^2_{b,c}], \label{DistTau} \\
\upsilon_{a,b,c}&=\langle\pi^1_{a,b}\otimes \mj_c,\pi^2_{a,b}\otimes \mj_c\rangle, \qquad \upsilon^{-1}_{a,b,c}=[\iota^1_{a,b}\otimes \mj_{c}, \iota^2_{a,b}\otimes \mj_{c}]. \label{DistUpsilon}
\end{align}
(We are aware that it is hard to distinguish between the Latin letter $v$, which is reserved for the isomorphism from $I^\ast$ to $I$ and the Greek letter $\upsilon$ denoting the isomorphism of the form $(a\oplus b)\otimes c\to (a\otimes c)\oplus (b\otimes c)$, but we decided to follow the notation from \cite{KL80} relevant for the strict compact closed structure, and from \cite{AC04} which is relevant for categorical quantum protocols.)

In a compact closed categories with biproducts, the \emph{scalars}, i.e.\ the endomorphisms from $I$ to $I$ form a commutative semiring ${\rm Hom}\,(I,I)$. The multiplication in this semiring is given by composition, for which $\mj_I$ is the neutral, and the addition is defined as above. (We will omit $\circ$ when we compose, i.e.\ multiply, scalars.) For a scalar $s\colon I\to I$ and an object $a$ of such a category, one defines the arrow $s_a\colon a\to a$ as the composition
\[
a\stackrel{\lambda^{-1}_a}{\longrightarrow}I\otimes a \stackrel{s\otimes a}{\longrightarrow}I\otimes a \stackrel{\lambda_a}{\longrightarrow} a,
\]
and the operation $s\bullet$ on arrows such that for $f\colon a\to b$, the arrow $s\bullet f$ is $f\circ s_a$. It is straightforward to check that this new operation satisfies the following equalities.
\begin{equation}\label{41}
   a\otimes (s\bullet f)=s\bullet (a\otimes f),\quad  (s\bullet f)\otimes a=s\bullet (f\otimes a),
\end{equation}
\begin{equation}\label{42}
   (s_2\bullet f_2)\circ(s_1\bullet f_1)=s_2s_1\bullet(f_2\circ f_1),
\end{equation}
\begin{equation}\label{43}
   \langle s\bullet f_1,\ldots, s\bullet f_n\rangle = s\bullet\langle f_1,\ldots,f_n\rangle.
\end{equation}

\begin{example}
As a paradigm for dagger compact closed category with dagger biproducts we use the category $\fdhilb$ of finite dimensional Hilbert spaces over the field $\mathbb{C}$ of complex numbers. The objects of this category are finite dimensional Hilbert spaces (finite dimensional vector spaces with inner product). The arrows of this category correspond to (bounded) linear maps between vector spaces. Dagger is given by the adjoint map. Since every vector space over $\mathbb{C}$ of dimension $n$ is isomorphic to $\mathbb{C}^n$, we can pass from $\fdhilb$ to its skeleton consisting of objects of the form  $\mathbb{C}^n$. By choosing orthogonal bases of such objects, the linear maps are envisaged as matrices. In this case, dagger corresponds to the usual adjoint of matrices (conjugation and transposition), and the operation $^*$ on arrows corresponds to the complex conjugation of an operator (matrix). Also, the operation $_*$ on arrows is given by a matrix transpose.
\end{example}

\section{The category $1\cob$}\label{jedancob}
The category $1\cob$ of 1-dimensional cobordisms has as objects closed oriented 0-dimensional manifolds i.e. finite (possibly empty) sequences of points together with their orientation (either $\ps$ or $\ms$). For example, an object of $1\cob$ is $\ps\ps\ms\ps\ms\ms$. Since there will be several roles of $\emptyset$ in this paper, we denote the empty sequence of points by $o$.

A compact oriented 1-dimensional topological manifold with boundary, i.e.\ a finite collection of oriented circles and line segments is called here 1-\emph{manifold}. For objects $a$ and $b$ of $1\cob$, a 1-\emph{cobordism} from $a$ to $b$ is a triple $(M,f_0\colon a\to M, f_1\colon b\to M)$, where $M$ is a 1-manifold and $f_0$, $f_1$ are embeddings. The boundary of $M$ is $\Sigma_0\coprod \Sigma_1$ and its orientation is induced from the orientation of $M$ (the initial point of an oriented segment is $\ps$ while the terminal is $\ms$). The embedding $f_0$ is orientation preserving and its image is $\Sigma_0$, while the embedding $f_1$ is orientation reversing and its image is $\Sigma_1$. Two cobordisms $(M,f_0,f_1)$ and $(M',f'_0,f'_1)$ from $a$ to $b$ are
\emph{equivalent}, when there is an
orientation preserving homeomorphism $F:M\to M'$ such that the
following diagram commutes.
\begin{center}
\begin{picture}(120,60)

\put(0,30){\makebox(0,0){$a$}}
\put(60,55){\makebox(0,0){$M$}} \put(60,5){\makebox(0,0){$M'$}}
\put(120,30){\makebox(0,0){$b$}}
\put(25,50){\makebox(0,0){$f_0$}}
\put(95,50){\makebox(0,0){$f_1$}}
\put(25,10){\makebox(0,0){$f'_0$}}
\put(95,10){\makebox(0,0){$f'_1$}} \put(67,30){\makebox(0,0){$F$}}

\put(10,35){\vector(2,1){40}} \put(10,25){\vector(2,-1){40}}
\put(110,35){\vector(-2,1){40}} \put(110,25){\vector(-2,-1){40}}
\put(60,45){\vector(0,-1){30}}
\end{picture}
\end{center}

The equivalence classes of 1-cobordisms are the \emph{arrows} of $1\cob$. The identity $\mj_a\colon a\to a$ is the equivalence class of $(a\times I, x\mapsto (x,0), x\mapsto (x,1))$, which in the case $a=o$ stands for the empty 1-cobordism from the empty sequence of oriented points $o$ to itself. Two cobordisms $(M,f_0,f_1)\colon a\to b$ and $(N,g_0,g_1)\colon b\to c$ are composed by ``gluing'', i.e.\ by making the pushout of $M\stackrel{f_1}{\longleftarrow} b\stackrel{g_0}{\longrightarrow} N$.
All the arrows of $1\cob$ are illustrated so that the source of an arrow is at the top, while its target is at the bottom of the picture. Therefore,the direction of pictures is \emph{top to bottom}, a convention used in \cite{PZ21}. Note that some authors use a different convection, \emph{left to right}, or \emph{bottom to top} \cite{K03,T10}. The latter is presumably the most popular in the physics literature.

The category $1\cob$ is dagger strict compact closed. We have symmetric monoidal structure on $1\cob$ in which $\otimes$ on objects is defined by concatenation, the empty sequence $o$ is the neutral and serves as the unit object $I$, while $\otimes$ on arrows is given by putting two cobordisms ``side by side''. The arrows $\alpha$ and $\lambda$ are identities and symmetry $\sigma$ is generated by transpositions $\ps\ps\to\ps\ps$, $\ps\ms\to\ms\ps$, $\ms\ps\to\ps\ms$ and $\ms\ms\to\ms\ms$. These transpositions are illustrated as follows:
\begin{center}
\begin{tikzpicture}[scale=0.75][line cap=round,line join=round,x=1.0cm,y=1.0cm]
\draw [line width=1pt](-2.0,3.0)-- (0.0,1.0);
\draw [line width=1pt](0.0,3.0)-- (-2.0,1.0);
\draw (-2.3,3.6) node[anchor=north west] {$\ps$};
\draw (-0.3,3.6) node[anchor=north west] {$\ps$};
\draw (-2.3,1) node[anchor=north west] {$\ps$};
\draw (-0.3,1) node[anchor=north west] {$\ps$};
\draw [->,>=stealth,line width=1.5pt] (-1.8620810895562505,1.1379189104437495) -- (-2.0,1.0);
\draw [->,>=stealth,line width=1.5pt] (-0.19407634363052395,1.194076343630524) -- (0.0,1.0);
\draw [line width=1pt](2.0,3.0)-- (4.0,1.0);
\draw [line width=1pt](4.0,3.0)-- (2.0,1.0);
\draw (1.7,3.6) node[anchor=north west] {$\ps$};
\draw (3.7,1) node[anchor=north west] {$\ps$};
\draw (3.7,3.6) node[anchor=north west] {$\ms$};
\draw (1.7,1) node[anchor=north west] {$\ms$};
\draw [->,>=stealth,line width=1.5pt] (3.7833534620239813,1.2166465379760187) -- (4.0,1.0);
\draw [->,>=stealth,line width=1.5pt] (3.836346125287993,2.836346125287993) -- (4.0,3.0);
\draw [line width=1pt](6.0,3.0)-- (8.0,1.0);
\draw [line width=1pt](8.0,3.0)-- (6.0,1.0);
\draw (5.7,1) node[anchor=north west] {$\ps$};
\draw (7.7,1) node[anchor=north west] {$\ms$};
\draw (5.7,3.6) node[anchor=north west] {$\ms$};
\draw (7.7,3.6) node[anchor=north west] {$\ps$};
\draw [->,>=stealth,line width=1.5pt] (6.128031727527586,1.1280317275275857) -- (6.0,1.0);
\draw [->,>=stealth,line width=1.5pt] (6.113497106927031,2.8865028930729686) -- (6.0,3.0);
\draw [line width=1pt](10.0,3.0)-- (12.0,1.0);
\draw [line width=1pt](12.0,3.0)-- (10.0,1.0);
\draw (9.7,1) node[anchor=north west] {$\ms$};
\draw (11.7,1) node[anchor=north west] {$\ms$};
\draw (9.7,3.6) node[anchor=north west] {$\ms$};
\draw (11.7,3.6) node[anchor=north west] {$\ms$};
\draw [->,>=stealth,line width=1.5pt] (10.106951097024632,2.8930489029753668) -- (10.0,3.0);
\draw [->,>=stealth,line width=1.5pt] (11.874531495701138,2.874531495701137) -- (12.0,3.0);
\begin{scriptsize}
\draw [fill=black] (-2.0,3.0) circle (2pt);
\draw [fill=black] (0.0,1.0) circle (2pt);
\draw [fill=black] (0.0,3.0) circle (2pt);
\draw [fill=black] (-2.0,1.0) circle (2pt);
\draw [fill=black] (2.0,3.0) circle (2pt);
\draw [fill=black] (4.0,1.0) circle (2pt);
\draw [fill=black] (4.0,3.0) circle (2pt);
\draw [fill=black] (2.0,1.0) circle (2pt);
\draw [fill=black] (6.0,3.0) circle (2pt);
\draw [fill=black] (8.0,1.0) circle (2pt);
\draw [fill=black] (8.0,3.0) circle (2pt);
\draw [fill=black] (6.0,1.0) circle (2pt);
\draw [fill=black] (10.0,3.0) circle (2pt);
\draw [fill=black] (12.0,1.0) circle (2pt);
\draw [fill=black] (12.0,3.0) circle (2pt);
\draw [fill=black] (10.0,1.0) circle (2pt);
\end{scriptsize}
\end{tikzpicture}
\end{center}
For example, the transposition $\ps\ms\to\ms\ps$ is a cobordism given by the manifold consisting of two oriented segments and two embeddings of the source $\ps\ms$ and the target $\ms\ps$ into its boundary (when a sign is mapped to the same sign, then the boundary point belongs to the source, and it is mapped to the opposite sign, then the boundary point belongs to the target of the cobordism).
\begin{center}
\begin{tikzpicture}[scale=0.7][line cap=round,line join=round,x=1.0cm,y=1.0cm]
\draw [line width=1pt](3.0,2.0)-- (3.0,0.0);
\draw [line width=1pt](5.0,2.0)-- (5.0,0.0);
\draw (2.7,2.6) node[anchor=north west] {$\ps$};
\draw (2.7,0) node[anchor=north west] {$\ms$};
\draw (4.7,0) node[anchor=north west] {$\ms$};
\draw (4.7,2.6) node[anchor=north west] {$\ps$};
\draw (2.7,3.6) node[anchor=north west] {$\ps$};
\draw (4.7,-1) node[anchor=north west] {$\ps$};
\draw (4.7,3.6) node[anchor=north west] {$\ms$};
\draw (2.7,-1) node[anchor=north west] {$\ms$};
\draw [shift={(3.027329841392542,2.501248065771235)}] plot[domain=3.6998468284590094:5.831693459929932,variable=\t]({-0.06663264376766559*0.5020183627985413*cos(\t r)+0.9977775758076203*0.3862846611334679*sin(\t r)},{-0.9977775758076203*0.5020183627985413*cos(\t r)+-0.06663264376766559*0.3862846611334679*sin(\t r)});
\draw [->,>=stealth,line width=0.7pt] (2.79,2.1160944273770226) -- (2.829066226905219,2.061767465840249);
\draw [shift={(1.8764548840937023,1.3418707322136805)}] plot[domain=2.4794310746680344:3.732998122805218,variable=\t]({-0.9972122563213643*4.092701151897252*cos(\t r)+0.0746171283449963*2.323155713193134*sin(\t r)},{-0.0746171283449963*4.092701151897252*cos(\t r)+-0.9972122563213643*2.323155713193134*sin(\t r)});
\draw [->,>=stealth,line width=0.7pt] (5.287621606183547,0.2340120686840863) -- (5.201805054089697,0.15836983701026686);
\draw [shift={(3.4946448612271457,-1.143013924723535)}] plot[domain=0.8659535543462812:1.9529583990254618,variable=\t]({0.9354274346052336*1.7893453813038553*cos(\t r)+0.3535187612954524*1.0222905047111188*sin(\t r)},{-0.3535187612954524*1.7893453813038553*cos(\t r)+0.9354274346052336*1.0222905047111188*sin(\t r)});
\draw [->,>=stealth,line width=0.7pt] (3.327272885622013,-0.03917712782000127) -- (3.2057656235563297,-0.01981913777619937);
\draw [shift={(3.453880241451742,1.6752339758056352)}] plot[domain=3.8372305438174075:4.920515525749326,variable=\t]({0.41929175490600235*3.0959978396251553*cos(\t r)+-0.9078515430773055*1.3307209979786425*sin(\t r)},{0.9078515430773055*3.0959978396251553*cos(\t r)+0.41929175490600235*1.3307209979786425*sin(\t r)});
\draw [->,>=stealth,line width=0.7pt] (4.868933944114557,1.6135375732301827) -- (4.9041345676088675,1.7100830685754715);
\draw [->,>=stealth,line width=1.5pt] (3.0,0.09737233802378534) -- (3.0,0.0);
\draw [->,>=stealth,line width=1.5pt] (5.0,0.13899264361832075) -- (5.0,0.0);
\begin{scriptsize}
\draw [fill=black] (3.0,2.0) circle (2pt);
\draw [fill=black] (3.0,0.0) circle (2pt);
\draw [fill=black] (5.0,2.0) circle (2pt);
\draw [fill=black] (5.0,0.0) circle (2pt);
\draw [fill=black] (3.0,3.0) circle (2pt);
\draw [fill=black] (5.0,3.0) circle (2pt);
\draw [fill=black] (3.0,-1.0) circle (2pt);
\draw [fill=black] (5.0,-1.0) circle (2pt);
\end{scriptsize}
\end{tikzpicture}
\end{center}

The dual $a^{\ast}$ of an object $a$ is the reversed sequence of points with reversed orientation. For example, if $a=\ps\ms\ms$, then $a^{\ast}=\ps\ps\ms$. (Note that this definition differs from the one given in \cite{PZ21} where just the orientation was reversed---both definitions are correct in presence of symmetry.) The arrows $\eta\colon o\to a^\ast\otimes a$ and $\varepsilon\colon a\otimes a^\ast\to o$, for $a$ as above are the following cobordisms:
\begin{center}	
\begin{tikzpicture}[scale=0.6] join=round,>=triangle 45,x=1.0cm,y=1.0cm]
\draw [line width=1pt][shift={(-4.5,0.0)}] plot[domain=0.0:3.141592653589793,variable=\t]({1.0*0.5*cos(\t r)+-0.0*0.5*sin(\t r)},{0.0*0.5*cos(\t r)+1.0*0.5*sin(\t r)});
\draw [line width=1pt][shift={(-4.5,0.0)}] plot[domain=0.0:3.141592653589793,variable=\t]({1.0*1.5*cos(\t r)+-0.0*1.5*sin(\t r)},{0.0*1.5*cos(\t r)+1.0*1.5*sin(\t r)});
\draw [line width=1pt][shift={(-4.5,0.0)}] plot[domain=0.0:3.141592653589793,variable=\t]({1.0*2.5*cos(\t r)+-0.0*2.5*sin(\t r)},{0.0*2.5*cos(\t r)+1.0*2.5*sin(\t r)});
\draw (-7.4,0) node[anchor=north west] {$\ps$};
\draw (-6.4,0) node[anchor=north west] {$\ps$};
\draw (-4.4,0) node[anchor=north west] {$\ps$};
\draw (-5.4,0) node[anchor=north west] {$\ms$};
\draw (-3.4,0) node[anchor=north west] {$\ms$};
\draw (-2.4,0) node[anchor=north west] {$\ms$};
\draw [->,>=stealth,line width=1.5pt] (-6.9894948221731,0.2289443827075123) -- (-7.0,0.0);
\draw [->,>=stealth,line width=1.5pt] (-5.986167605641045,0.20323840174328178) -- (-6.0,0.0);
\draw [->,>=stealth,line width=1.5pt] (-4.03,0.08985219229327004) -- (-4.0,0.0);
\draw [line width=1pt][shift={(4.498549448749316,2.48908093738758)}] plot[domain=3.141592653589793:6.283185307179586,variable=\t]({1.0*0.5000000000000004*cos(\t r)+-0.0*0.5000000000000004*sin(\t r)},{0.0*0.5000000000000004*cos(\t r)+1.0*0.5000000000000004*sin(\t r)});
\draw [line width=1pt][shift={(4.498549448749316,2.48908093738758)}] plot[domain=3.141592653589793:6.283185307179586,variable=\t]({1.0*1.5*cos(\t r)+-0.0*1.5*sin(\t r)},{0.0*1.5*cos(\t r)+1.0*1.5*sin(\t r)});
\draw [line width=1pt][shift={(4.498549448749316,2.48908093738758)}] plot[domain=3.141592653589793:6.283185307179586,variable=\t]({1.0*2.5000000000000004*cos(\t r)+-0.0*2.5000000000000004*sin(\t r)},{0.0*2.5000000000000004*cos(\t r)+1.0*2.5000000000000004*sin(\t r)});
\draw (1.55,3.2) node[anchor=north west] {$\ps$};
\draw (4.55,3.2) node[anchor=north west] {$\ps$};
\draw (5.55,3.2) node[anchor=north west] {$\ps$};
\draw (2.55,3.2) node[anchor=north west] {$\ms$};
\draw (3.55,3.2) node[anchor=north west] {$\ms$};
\draw (6.55,3.2) node[anchor=north west] {$\ms$};
\draw [->,>=stealth,line width=1.5pt] (3.003901328197095,2.3624831659628196) -- (2.9985494487493156,2.48908093738758);
\draw [->,>=stealth,line width=1.5pt] (4.02,2.412304051157032) -- (3.998549448749315,2.48908093738758);
\draw [->,>=stealth,line width=1.5pt] (6.996052909003227,2.377382835256246) -- (6.998549448749316,2.48908093738758);
\draw (-4.8,3.2) node[anchor=north west] {$o$};
\draw (4.3,0) node[anchor=north west] {$o$};
\begin{scriptsize}
\draw [fill=black] (-7.0,0.0) circle (2pt);
\draw [fill=black] (-6.0,0.0) circle (2pt);
\draw [fill=black] (-5.0,0.0) circle (2pt);
\draw [fill=black] (-4.0,0.0) circle (2pt);
\draw [fill=black] (-3.0,0.0) circle (2pt);
\draw [fill=black] (-2.0,0.0) circle (2pt);
\draw [fill=black] (1.9985494487493152,2.48908093738758) circle (2pt);
\draw [fill=black] (2.9985494487493156,2.48908093738758) circle (2pt);
\draw [fill=black] (3.998549448749315,2.48908093738758) circle (2pt);
\draw [fill=black] (4.998549448749316,2.48908093738758) circle (2pt);
\draw [fill=black] (5.998549448749316,2.48908093738758) circle (2pt);
\draw [fill=black] (6.998549448749316,2.48908093738758) circle (2pt);
\end{scriptsize}
\end{tikzpicture}
\end{center}
It is not difficult to check that the arrows $u_{a,b}$, $v$ and $w_{a}$ are all identities.

Let $f\colon a\to b$ be an arrow of $1\cob$ represented by a triple $(M,f_0\colon a\to M, f_1\colon b\to M)$. Its name $\ulcorner f \urcorner\colon o\to a^\ast\otimes b$ is represented by the triple $(M,g_0\colon o\to M, g_1\colon a^\ast\otimes b\to M)$, where for every point $x$ of $a$ and the corresponding point $\bar{x}$ of $a^\ast$ we have $g_1(\bar{x})=f_0(x)$ and for every point $y$ of $b$ we have $g_1(y)=f_1(y)$. The coname $\llcorner f \lrcorner$ of $f$ is defined in $1\cob$ analogously.
\begin{center}
\begin{tikzpicture}[scale=0.4,line cap=round,line join=round,x=1.0cm,y=1.0cm]
\draw [line width=1pt][shift={(0.0,0.0)}] plot[domain=0.0:3.141592653589793,variable=\t]({1.0*1.0*cos(\t r)+-0.0*1.0*sin(\t r)},{0.0*1.0*cos(\t r)+1.0*1.0*sin(\t r)});
\draw [line width=1pt](3.0,0.0)-- (1.0,3.0);
\draw [line width=1pt][shift={(11.0,0.0)}] plot[domain=0.0:3.141592653589793,variable=\t]({1.0*3.0*cos(\t r)+-0.0*3.0*sin(\t r)},{0.0*3.0*cos(\t r)+1.0*3.0*sin(\t r)});
\draw [line width=1pt][shift={(11.0,0.0)}] plot[domain=0.0:3.141592653589793,variable=\t]({1.0*1.0*cos(\t r)+-0.0*1.0*sin(\t r)},{0.0*1.0*cos(\t r)+1.0*1.0*sin(\t r)});

\draw [line width=1pt][shift={(20.0,3.0)}] plot[domain=3.141592653589793:6.283185307179586,variable=\t]({1.0*1.0*cos(\t r)+-0.0*1.0*sin(\t r)},{0.0*1.0*cos(\t r)+1.0*1.0*sin(\t r)});

\draw [line width=1pt][shift={(24.0,3.0)}] plot[domain=3.141592653589793:6.283185307179586,variable=\t]({1.0*1.0*cos(\t r)+-0.0*1.0*sin(\t r)},{0.0*1.0*cos(\t r)+1.0*1.0*sin(\t r)});
\draw (-1.5,0) node[anchor=north west] {$\ps$};
\draw (0.5,0) node[anchor=north west] {$\ms$};
\draw (2.5,0) node[anchor=north west] {$\ps$};
\draw (0.4,4) node[anchor=north west] {$\ps$};
\draw (7.5,0) node[anchor=north west] {$\ms$};
\draw (9.5,0) node[anchor=north west] {$\ps$};
\draw (11.5,0) node[anchor=north west] {$\ms$};
\draw (20.5,4) node[anchor=north west] {$\ms$};
\draw (24.5,4) node[anchor=north west] {$\ms$};
\draw (13.5,0) node[anchor=north west] {$\ps$};
\draw (18.5,4) node[anchor=north west] {$\ps$};
\draw (22.5,4) node[anchor=north west] {$\ps$};
\draw (-3,2) node[anchor=north west] {$f\colon$};
\draw (5.5,2) node[anchor=north west] {$\ulcorner f \urcorner\colon$};
\draw (16.5,2) node[anchor=north west] {$\llcorner f \lrcorner\colon$};
\draw (10.5,4) node[anchor=north west] {$o$};
\draw (21.5,0) node[anchor=north west] {$o$};
\draw [->,>=stealth][line width=1.3pt] (13.99078715835982,0.23493014576675697) -- (14.0,0.0);
\draw [->,>=stealth][line width=1.3pt] (10.03,0.1557154478588668) -- (10.0,0.0);
\draw [->,>=stealth][line width=1.3pt] (24.976,2.8828994573309963) -- (25.0,3.0);
\draw [->,>=stealth][line width=1.3pt] (20.976,2.883124910029953) -- (21.0,3.0);
\draw [->,>=stealth][line width=1.3pt] (-0.97,0.14968431428115275) -- (-1.0,0.0);
\draw [->,>=stealth][line width=1.3pt] (2.9023037670756007,0.1465443493865992) -- (3.0,0.0);
\draw [fill=black] (-1.0,0.0) circle (3pt);
\draw [fill=black] (1.0,0.0) circle (3pt);
\draw [fill=black] (3.0,0.0) circle (3pt);	
\draw [fill=black] (1.0,3.0) circle (3pt);
\draw [fill=black] (8.0,0.0) circle (3pt);
\draw [fill=black] (14.0,0.0) circle (3pt);
\draw [fill=black] (10.0,0.0) circle (3pt);
\draw [fill=black] (12.0,0.0) circle (3pt);
\draw [fill=black] (19.0,3.0) circle (3pt);
\draw [fill=black] (25.0,3.0) circle (3pt);
\draw [fill=black] (23.0,3.0) circle (3pt);
\draw [fill=black] (21.0,3.0) circle (3pt);
\end{tikzpicture}
\end{center}
The arrow $f^\ast\colon b^\ast\to a^\ast$ is represented by the triple $(M,h_0\colon b^\ast\to M, h_1\colon a^\ast\to M)$, where for every point $x$ of $a$ and the corresponding point $\bar{x}$ of $a^\ast$ we have $h_1(\bar{x})=f_0(x)$, and for every point $y$ of $b$ and the corresponding point $\bar{y}$ of $b^\ast$ we have $h_0(\bar{y})=f_1(y)$.

The cobordism $f^\dagger\colon b\to a$ is obtained by reversing the orientation of the 1-manifold representing the cobordism $f\colon a\to b$. It is not hard to check that the equalities~\ref{28}-\ref{30} of Appendix~\ref{language} hold.
\begin{center}
\begin{tikzpicture}[scale=0.4,line cap=round,line join=round,x=1.0cm,y=1.0cm]
\draw [line width =1pt][shift={(2.0,-2.0)}] plot[domain=3.141592653589793:6.283185307179586,variable=\t]({1.0*1.0*cos(\t r)+-0.0*1.0*sin(\t r)},{0.0*1.0*cos(\t r)+1.0*1.0*sin(\t r)});
\draw [line width =1pt](-1.0,-2.0)-- (1.0,-5.0);
\draw [line width =1pt][shift={(10.0,-2.0)}] plot[domain=3.141592653589793:6.283185307179586,variable=\t]({1.0*1.0*cos(\t r)+-0.0*1.0*sin(\t r)},{0.0*1.0*cos(\t r)+1.0*1.0*sin(\t r)});
\draw [line width =1pt](13.0,-2.0)-- (11.0,-5.0);
\draw [line width =1pt](21.0,-2.0)-- (19.0,-5.0);
\draw [line width =1pt][shift={(22.0,-5.0)}] plot[domain=0.0:3.141592653589793,variable=\t]({1.0*1.0*cos(\t r)+-0.0*1.0*sin(\t r)},{0.0*1.0*cos(\t r)+1.0*1.0*sin(\t r)});
\draw (-1.6,-1) node[anchor=north west] {$\ms$};
\draw (2.4,-1) node[anchor=north west] {$\ms$};
\draw (0.6,-5) node[anchor=north west] {$\ms$};
\draw (10.4,-1) node[anchor=north west] {$\ms$};
\draw (20.5,-1) node[anchor=north west] {$\ms$};
\draw (18.4,-5) node[anchor=north west] {$\ms$};
\draw (22.5,-5) node[anchor=north west] {$\ms$};
\draw (10.4,-5) node[anchor=north west] {$\ps$};
\draw (0.4,-1) node[anchor=north west] {$\ps$};
\draw (8.4,-1) node[anchor=north west] {$\ps$};
\draw (12.4,-1) node[anchor=north west] {$\ps$};
\draw (20.5,-5) node[anchor=north west] {$\ps$};
\draw [->,>=stealth][line width=1.3pt] (-0.8548899609460593,-2.217665058580911) -- (-1.0,-2.0);
\draw [->,>=stealth][line width=1.3pt] (2.988,-2.05411) -- (3,-2);
\draw [->,>=stealth][line width=1.3pt] (10.975,-2.1132417510522834) -- (11.0,-2.0);
\draw [->,>=stealth][line width=1.3pt] (11.080873787907723,-4.878689318138414) -- (11.0,-5.0);
\draw [->,>=stealth][line width=1.3pt] (20.9211413589203,-2.118287961619552) -- (21.0,-2.0);
\draw [->,>=stealth][line width=1.3pt] (21.03,-4.872511662323953) -- (21.0,-5.0);
\draw (-3,-3) node[anchor=north west] {$f^\ast\colon$};
\draw (7,-3) node[anchor=north west] {$f^\dagger\colon$};
\draw (17,-3) node[anchor=north west] {$f_\ast\colon$};
\draw [fill=black] (3.0,-2.0) circle (3pt);
\draw [fill=black] (1.0,-2.0) circle (3pt);
\draw [fill=black] (-1.0,-2.0) circle (3pt);
\draw [fill=black] (1.0,-5.0) circle (3pt);
\draw [fill=black] (11.0,-2.0) circle (3pt);
\draw [fill=black] (9.0,-2.0) circle (3pt);
\draw [fill=black] (13.0,-2.0) circle (3pt);
\draw [fill=black] (11.0,-5.0) circle (3pt);
\draw [fill=black] (19.0,-5.0) circle (3pt);
\draw [fill=black] (21.0,-5.0) circle (3pt);
\draw [fill=black] (23.0,-5.0) circle (3pt);
\draw [fill=black] (21.0,-2.0) circle (3pt);
\end{tikzpicture}
\end{center}

By the above definitions of $f^\ast$ and $f^\dagger$ for a cobordism $f\colon a\to b$ it is straightforward to reconstruct the cobordism $f_\ast=(f^\dagger)^\ast\colon a^\ast\to b^\ast$.

\section{A pair of free categories}\label{free category}

We start with a construction of a dagger compact closed category $\mathcal{F}^\dagger$ with dagger biproducts freely generated by a single object $p$ and a set $\Gamma$ of unitary endomorphisms on this object. An arrow $f\colon a\to b$ in a dagger category is \emph{unitary} when $f^\dagger\colon b\to a$ is its both-sided inverse. The universal property of $\mathcal{F}^\dagger$ is the following: for every dagger compact closed category $\mathcal{C}$ with dagger biproducts, and a function $\varphi$ from the set $\Gamma$ to the set of unitary endomorphisms of an object $c$ of $\mathcal{C}$, there exists a unique functor $F\colon \mathcal{F}^\dagger \to \mathcal{C}$ strictly preserving the whole structure, such that $Fp=c$ and for every $\gamma\in \Gamma$, $F\gamma=\varphi(\gamma)$.

Our construction of this category is syntactical; it is akin to the construction of the category $F\mathcal{A}$ from \cite[Sections 3-4]{KL80}, and it follows the construction of the category $\mathcal{F}_P$ from \cite[Section 4]{PZ21}. As noted in \cite{KL80}, it is ``perfectly general, applying to categories with any explicitly-given equational extra structure''. The \emph{objects} of $\mathcal{F}^\dagger$ are the formulae built out of a single letter $p$ and the constants $I$ and $0$, with the help of one unary connective $\ast$ (written as a superscript) and two binary connectives $\otimes$ and $\oplus$.

The arrows of $\mathcal{F}^\dagger$ are obtained as equivalence classes of \emph{terms} built in the following manner. We start with \emph{primitive terms}, which are of the form $\gamma$, $\gamma^{-1}$ for every $\gamma\in\Gamma$, or $\mj_a$, $\alpha_{a,b,c}$, $\alpha^{-1}_{a,b,c}$, $\lambda_a$, $\lambda^{-1}_a$, $\sigma_{a,b}$, $\eta_a$, $\varepsilon_a$, $\pi^1_{a,b}$, $\pi^2_{a,b}$, $\iota^1_{a,b}$, $\iota^2_{a,b}$ and $0_{a,b}$, for all objects $a$, $b$ and $c$ of  $\mathcal{F}^\dagger$. The \emph{terms} are built out of primitive terms with the help of one unary operational symbol $\dagger$ and four binary operational symbols $\otimes$, $\oplus$, $+$ and $\circ$. (Each such term is equipped with the source and the target, which are objects of $\mathcal{F}^\dagger$, e.g.\ the source and the target of every $\gamma\in\Gamma$ is $p$, and constructions of terms with $+$ and $\circ$ are restricted to appropriate sources and targets.) The equivalence classes of these terms, i.e.\ the \emph{arrows} of $\mathcal{F}^\dagger$, are obtained modulo the congruence generated by the equalities \ref{1}-\ref{31} of Appendix~\ref{language} and, for every $\gamma\in\Gamma$, the equalities \ref{33}-\ref{34} below.
\begin{equation}\label{33}
   \gamma\circ\gamma^{-1}=\mj_p=\gamma^{-1}\circ\gamma,
\end{equation}
\begin{equation}\label{34}
   \gamma^\dagger=\gamma^{-1}.
\end{equation}

On the other hand, consider the category $\mathcal{F}$ with the same objects and the same primitive terms as $\mathcal{F}^\dagger$, just the terms of $\mathcal{F}$ are constructed without the unary operational symbol $\dagger$. The \emph{arrows} of $\mathcal{F}$, are the equivalence classes of these terms, modulo the congruence generated by the equalities \ref{1}-\ref{22} and \ref{33}. The category $\mathcal{F}$ is a compact closed category with biproducts freely generated by the group (envisaged as a category with one object) freely generated by the set $\Gamma$. The universal property of $\mathcal{F}$ is the following: for every compact closed category $\mathcal{C}$ with biproducts, and a function $\varphi$ from the set $\Gamma$ to the set of automorphisms of an object $c$ of $\mathcal{C}$, there exists a unique functor $F\colon \mathcal{F} \to \mathcal{C}$ strictly preserving the whole structure, such that $Fp=c$ and for every $\gamma\in \Gamma$, $F\gamma=\varphi(\gamma)$.

\begin{prop}\label{free-iso}
The categories $\mathcal{F}^\dagger$ and $\mathcal{F}$ are isomorphic.
\end{prop}
\begin{proof}
From the equalities \ref{4}-\ref{36} it follows that every arrow of $\mathcal{F}^\dagger$ (as an equivalence class of terms) contains a $\dagger$-free term. Also, every equality assumed for $\mathcal{F}^\dagger$ in which $\dagger$ appears boils down to the trivial identity after $\dagger$-elimination at both sides. Thus, the identity on objects and the function on arrows that maps the equivalence class of a term in $\mathcal{F}$ to the equivalence class of the same term in $\mathcal{F}^\dagger$ is an isomorphism between these two categories.
\end{proof}

\section{Injections and projections}\label{injekcije}

For the functor $^\ast\colon\mathcal{F}^{\mathrm{op}}\to \mathcal{F}$ defined as in Section~\ref{closed}, the unit $\eta$ and the counit $\varepsilon$ become \textit{dinatural}, i.e. for $f\colon a\to b$ the following equalities hold:
\begin{equation} \label{Jednakost:Diprirodnost}
(a^\ast\otimes f)\circ \eta_a = (f^\ast\otimes b)\circ \eta_b,  \qquad \varepsilon_a\circ (a\otimes f^{\ast}) = \varepsilon_b\circ (f\otimes b^\ast).
\end{equation}
Also, for arrows $f,g:a\to b$ in $\mathcal{F}$, the following equality holds,
\begin{equation} \label{Jedn:ZvKrozPlus}
(f+g)^{\ast}=f^{\ast}+g^{\ast}.
\end{equation}

\begin{definition}\label{proj-inj}
Let $a$ be an object of $\mathcal{F}$. By induction on complexity of $a$, we define two finite sequences ${\rm I}_a=(\iota^0_a,\ldots,\iota^{n-1}_a)$ (the \emph{injections} of $a$) and $\Pi_a=(\pi^0_a,\ldots,\pi^{n-1}_a)$ (the \emph{projections} of $a$) of arrows of $\mathcal{F}$ in the following way. If $a$ is the letter $p$ or either  $I$ or $0$, then $n=1$ and ${\rm I}_a=(\mj_a)=\Pi_a$. Let us assume that ${\rm I}_{a_1}=(\iota^0_1,\ldots,\iota^{n_1-1}_1)$, $\Pi_{a_1}=(\pi^0_1,\ldots,\pi^{n_1-1}_1)$ and ${\rm I}_{a_2}=(\iota^0_2,\ldots,\iota^{n_2-1}_2)$, $\Pi_{a_2}=(\pi^0_2,\ldots,\pi^{n_2-1}_2)$ are already defined. For $\lfloor x\rfloor$ being the \emph{floor} function of a real $x$, i.e. the greatest integer less than or equal to $x$, and $i\text{ mod}\, n$ being the residue of $i$ modulo $n$, we have the following.
\begin{itemize}
\item[$\otimes$] If $a=a_1\otimes a_2$, then $n=n_1\cdot n_2$, and for $0\leq i< n_1\cdot n_2$,
    \[
    \iota^i_a=\iota^{\lfloor i/n_2\rfloor}_1\otimes \iota^{i\text{ mod}\, n_2}_2,\quad\quad \pi^i_a=\pi^{\lfloor i/n_2\rfloor}_1\otimes \pi^{i\text{ mod}\, n_2}_2.
    \]

\item[$\ast$] If $a=a_1^\ast$, then $n=n_1$, and for $0\leq i< n_1$,
    $$\iota^i_a=(\pi^i_1)^\ast, \quad \pi^i_a=(\iota^i_1)^\ast.$$

\item[$\oplus$] If $a=a_1\oplus a_2$, then $n=n_1+ n_2$, and for $0\leq i< n_1+ n_2$, $s_i=\left\lfloor\frac{\min\{i, n_1\}}{n_1}\right\rfloor$
    \[ \iota^i_a= \left\{\begin{array}{ll}
    \iota^1_{a_1,a_2}\circ \iota^i_1, & 0\leq i<n_1,
    \\[1ex]
    \iota^2_{a_1,a_2}\circ \iota^{i-n_1}_2, & \mbox{\rm otherwise}, \end{array} \right.
    =\iota^{1+s_i}_{a_1,a_2}\circ\iota^{i-n_1\cdot s_i}_{1+s_i},
    \]

    \[
    \pi^i_a= \left\{\begin{array}{ll}
    \pi^i_1\circ \pi^1_{a_1,a_2}, & 0\leq i<n_1,
    \\[1ex]
    \pi^{i-n_1}_2\circ \pi^2_{a_1,a_2}, & \mbox{\rm otherwise}. \end{array} \right.
    = \pi^{i-n_1\cdot s_i}_{1+s_i}\circ \pi^{1+s_i}_{a_1,a_2}.
\]
\end{itemize}
\end{definition}

\begin{example} Let $a=(p\oplus I)\oplus 0$ and $b=((p\oplus 0)\oplus p)\otimes (I\oplus p)^{\ast}$. Then $\iota_a^i$ and $\pi_a^i$ for $0\leq i<3$ as well as $\iota_b^j$ and $\pi_b^j$ for $0\leq j<6$ are given in the following tables.
\begin{table}[ht]
\begin{minipage}{0.28\linewidth}\centering
\begin{TAB}[3pt]{|c|c|}{|c|c|c|c|c|c|}
$\iota_{a}^0$ & $\iota_{p\oplus I,0}^1\circ \iota^1_{p,I}$  \\
$\iota_{a}^1$ & $\iota_{p\oplus I,0}^1\circ \iota^2_{p,I}$  \\
$\iota_{a}^2$ & $\iota_{p\oplus I,0}^2$  \\
$\pi_{a}^0$ & $\pi_{p,I}^1\circ \pi^1_{p\oplus I,0}$  \\
$\pi_{a}^1$ & $\pi_{p,I}^2\circ \pi^1_{p\oplus I,0}$  \\
$\pi_{a}^2$ & $\pi_{p\oplus I,0}^2$  \\
\end{TAB}
\end{minipage}
\begin{minipage}{0.71\linewidth}\centering
\begin{TAB}[3pt]{|c|c|c|c|}{|c|c|c|c|c|c|}
$\iota_{b}^0$ & $(\iota_{p\oplus 0,p}^1\circ \iota^1_{p,0})\otimes (\pi_{I,p}^1)^{\ast}$ & $\pi_{b}^0$ & $(\pi_{p,0}^1\circ \pi^1_{p\oplus 0,p})\otimes (\iota_{I,p}^1)^{\ast}$ \\
$\iota_{b}^1$ & $(\iota_{p\oplus 0,p}^1\circ \iota^1_{p,0})\otimes (\pi_{I,p}^2)^{\ast}$ & $\pi_{b}^1$ & $(\pi_{p,0}^1\circ \pi^1_{p\oplus 0,p})\otimes (\iota_{I,p}^2)^{\ast}$ \\
$\iota_{b}^2$ & $(\iota_{p\oplus 0,p}^1\circ \iota^2_{p,0})\otimes (\pi_{I,p}^1)^{\ast}$ & $\pi_{b}^2$ & $(\pi_{p,0}^2 \circ \pi^1_{p\oplus 0,p})\otimes (\iota_{I,p}^1)^{\ast}$ \\
$\iota_{b}^3$ & $(\iota_{p\oplus 0,p}^1\circ \iota^2_{p,0})\otimes (\pi_{I,p}^2)^{\ast}$ & $\pi_{b}^3$ & $(\pi_{p,0}^2 \circ \pi^1_{p\oplus 0,p})\otimes (\iota_{I,p}^2)^{\ast}$ \\
$\iota_{b}^4$ & $\iota^2_{p\oplus 0,p}\otimes (\pi^1_{I,p})^{\ast}$ & $\pi_{b}^4$ & $\pi_{p\oplus 0,p}^2\otimes (\iota_{I,p}^1)^{\ast}$ \\
$\iota_{b}^5$ & $\iota^2_{p\oplus 0,p}\otimes (\pi^2_{I,p})^{\ast}$ & $\pi_{b}^5$ & $\pi_{p\oplus 0,p}^2\otimes (\iota_{I,p}^2)^{\ast}$ \\
\end{TAB}
\end{minipage}
\end{table}
\end{example}

\begin{rem}\label{3.1}
For every $0\leq i<n$, the target of $\iota^i_a$ and the source of $\pi^i_a$ are both equal to $a$, while the source $a^i$ of $\iota^i_a$ is equal to the target of $\pi^i_a$, and $a^i$ is $\oplus$-free. Moreover, if $a$ is $\oplus$-free, then ${\rm I}_a=(\mj_a)=\Pi_a$.
\end{rem}

The following proposition establishes the desired properties of injections and projections.

\begin{prop}\label{Prop:InjProj}
For every object $a$ of $\mathcal{F}$
\[
\pi^j_a\circ\iota^i_a= \left\{
\begin{array}{ll} \mj_{a^i}, & i=j,\\[1ex]
0_{a^i,a^j}, & \mbox{\rm otherwise},
\end{array} \right . \quad\quad\quad \sum_{i=0}^{n-1} \iota^i_a\circ\pi^i_a =\mj_a.
\]
\end{prop}
\begin{proof}
We proceed by induction on complexity of $a$. When $a$ is $p,I$ or $0$, all injections and projections are identities, and the claim holds. For the inductive step, we consider the following three cases.

\medskip
\noindent (1) Suppose that $a=a_1\otimes a_2$, where $|I_{a_1}|=|\Pi_{a_1}|=n_1$ and  $|I_{a_2}|=|\Pi_{a_2}|=n_2$. Then we have
\begin{align*}
\pi_a^j\circ \iota_a^i &= (\pi_1^{\lfloor \frac{j}{n_2}\rfloor} \otimes \pi_2^{j\;\mathrm{mod}\;n_2})\circ (\iota_1^{\lfloor \frac{i}{n_2}\rfloor} \otimes \iota_2^{i\;\mathrm{mod}\;n_2}) \\
&=(\pi_1^{\lfloor \frac{j}{n_2}\rfloor} \circ \iota_1^{\lfloor \frac{i}{n_2}\rfloor})\otimes (\pi_2^{j\;\mathrm{mod}\;n_2}\circ \iota_2^{i\;\mathrm{mod}\;n_2}),
\end{align*}
and the first claim follows by the inductive hypothesis. Also, using the inductive hypothesis and the equality \ref{25} of Appendix~\ref{language}, we have
\begin{align*}
\sum_{i=0}^{n-1} \iota_a^i\circ\pi_a^i &=
\sum_{i=0}^{n-1} (\iota_1^{\lfloor \frac{i}{n_2}\rfloor} \otimes \iota_2^{i\;\mathrm{mod}\;n_2}) \circ (\pi_1^{\lfloor \frac{i}{n_2}\rfloor} \otimes \pi_2^{i\;\mathrm{mod}\;n_2}) \\
&= \sum_{i=0}^{n-1} (\iota_1^{\lfloor \frac{i}{n_2}\rfloor} \circ \pi_1^{\lfloor \frac{i}{n_2}\rfloor})\otimes (\iota_2^{i\;\mathrm{mod}\;n_2}\circ \pi_2^{i\;\mathrm{mod}\;n_2}) \\
&=\sum_{k=0}^{n_1-1} \left( (\iota_1^k\circ \pi_1^k) \otimes \sum_{l=0}^{n_2-1}(\iota_2^l\circ \pi_2^l) \right) = \sum_{k=0}^{n_1-1} (\iota_1^k\circ \pi_1^k)\otimes a_2 \\
&= \left(\sum_{k=0}^{n_1-1} (\iota_1^k\circ \pi_1^k)\right) \otimes a_2=\mathbf{1}_{a_1}\otimes a_2=\mathbf{1}_{a_1\otimes a_2}.
\end{align*}

\medskip
\noindent (2) When $a=a_1^\ast$, we have
\begin{align*}
\pi_a^j\circ \iota_a^i &= (\iota_1^j)^{\ast}\circ (\pi_1^i)^{\ast}=(\pi_1^i\circ \iota_1^j)^{\ast},
\end{align*}
and the first claim follows by the inductive hypothesis. Also, using the inductive hypothesis and the equality \ref{Jedn:ZvKrozPlus}, we have
\begin{align*}
\sum_{i=0}^{n-1} \iota_a^i\circ\pi_a^i &= \sum_{i=0}^{n-1} (\pi_1^i)^{\ast}\circ (\iota_1^i)^{\ast}= \sum_{i=0}^{n-1} (\iota_1^i\circ \pi_1^i)^{\ast}= \left(\sum_{i=0}^{n-1} \iota_1^i\circ \pi_1^i\right)^{\ast}=(\mj_{a_1})^{\ast}=\mj_a.
\end{align*}

\medskip
\noindent (3) Suppose that $a=a_1\oplus a_2$, and again $|I_{a_1}|=|\Pi_{a_1}|=n_1$ and  $|I_{a_2}|=|\Pi_{a_2}|=n_2$. We have
$$\pi_a^j\circ\iota_a^i=\pi_{1+s_j}^{j-n_1\cdot s_j}\circ \pi_{a_1,a_2}^{1+s_j}\circ \iota_{a_1,a_2}^{1+s_i}\circ \iota_{1+s_i}^{i-n_1\cdot s_i}.$$
Since $i\neq j$ implies $1+s_i\neq 1+s_j$ or $i-n_1\cdot s_i\neq j-n_1\cdot s_j$, the first claim follows according to the inductive hypothesis. For the second claim, we have
\begin{align*}
\sum_{i=0}^{n-1} \iota_a^i\circ \pi_a^i &= \sum_{i=0}^{n-1} \iota_{a_1,a_2}^{1+s_i}\circ \iota_{1+s_i}^{i-n_1\cdot s_i}\circ \pi_{1+s_i}^{i-n_1\cdot s_i}\circ \pi_{a_1,a_2}^{1+s_i} \\
&= \sum_{i=0}^{n_1-1}\iota_{a_1,a_2}^1\circ \iota_{1}^i\circ \pi_1^i\circ \pi_{a_1,a_2}^1 + \sum_{j=0}^{n_2-1}\iota_{a_1,a_2}^2\circ
\iota_{2}^j\circ \pi_2^j\circ \pi_{a_1,a_2}^2 \\
&= \iota_{a_1,a_2}^1\circ \left(\sum_{i=0}^{n_1-1} \iota_{1}^i\circ \pi_1^i\right )\circ \pi_{a_1,a_2}^1 + \iota_{a_1,a_2}^2\circ \left(\sum_{j=0}^{n_2-1}
\iota_{2}^j\circ \pi_2^j\right) \circ\pi_{a_1,a_2}^2 \\
&= \iota_{a_1,a_2}^1\circ\pi_{a_1,a_2}^1+\iota_{a_1,a_2}^2\circ\pi_{a_1,a_2}^2=
\mathbf{1}_{a_1\oplus a_2}. \qedhere
\end{align*}
\end{proof}

As a corollary of Proposition~\ref{Prop:InjProj} we have the following.

\begin{cor}\label{3.3}
For every object $a$ of $\mathcal{F}$, the cocone $(a,{\rm I}_a)$ together with the cone $(a,\Pi_a)$ make a biproduct.
\end{cor}
\begin{proof}
In order to show that the cocone $(a,{\rm I}_a)$ is universal, consider for $0\leq i<n$ arrows $f^i\colon a^i\to c$ of $\mathcal{F}$ and define $h\colon a\to c$ to be $\sum_{i=0}^{n-1} f^i\circ\pi^i_a$. For every $0\leq i<n$, by the left-hand side equality of Proposition~\ref{Prop:InjProj}, we have that $h$ satisfies $h\circ\iota^i_a=f^i$. Assume that $h'\colon a\to c$ for every $0\leq i<n$ satisfies $h'\circ\iota^i_a=f^i$. We conclude that
\[
h'\circ \sum_{i=0}^{n-1} \iota^i_a\circ\pi^i_a=\sum_{i=0}^{n-1} f^i\circ\pi^i_a=h,
\]
and by the right-hand side equality of Proposition~\ref{Prop:InjProj}, we have $h'=h$. That $(a,\Pi_a)$ is a universal cone is proved analogously.
\end{proof}

\section{A normalisation}\label{normalizacija}

Our normalisation of arrows of the category $\mathcal{F}$ is a procedure derived from the one developed in \cite[Section~5]{PZ21}.
The goal is to represent every arrow of $\mathcal{F}$, whose source and target are $\oplus$-free, by a term free of occurrences of $\oplus$, $\iota$ and $\pi$.

For every arrow $u\colon a\to b$ of $\mathcal{F}$, where ${\rm I}_a= (\iota^0_a, \ldots,\iota^{n-1}_a)$, $\Pi_b=(\pi^0_b,\ldots,\pi^{m-1}_b)$, let $M_u$ be the $m\times n$ matrix whose $ij$-entry is $\pi^i_b\circ u\circ \iota^j_a\colon a^j\to b^i$.
Let $X$ be an $m\times n$ matrix whose $ij$-entry is an arrow of $\mathcal{F}$ from $a^j$ to $b^i$ and let $Y$ be a $q\times r$  matrix whose $ij$-entry is an arrow of $\mathcal{F}$ from $c^j$ to $d^i$. We define $X\otimes Y$ as the Kronecker product of matrices over a field, save that the multiplication in the field is replaced by the tensor product of arrows in $\mathcal{F}$. For example,
\[
\left(\begin{array}{ccc} x_{00} & x_{01} & x_{02}
\\
x_{10} & x_{11} & x_{12} \end{array}\right) \otimes \left(\begin{array}{cc} y_{00} & y_{01}
\\
y_{10} & y_{11} \end{array}\right)
\]
is
\[
\left(\begin{array}{cccccc}
x_{00}\otimes y_{00} & x_{00}\otimes y_{01} & x_{01}\otimes y_{00} & x_{01}\otimes y_{01} & x_{02}\otimes y_{00} & x_{02}\otimes y_{01}
\\
x_{00}\otimes y_{10} & x_{00}\otimes y_{11} & x_{01}\otimes y_{10} & x_{01}\otimes y_{11} & x_{02}\otimes y_{10} & x_{02}\otimes y_{11}
\\
x_{10}\otimes y_{00} & x_{10}\otimes y_{01} & x_{11}\otimes y_{00} & x_{11}\otimes y_{01} & x_{12}\otimes y_{00} & x_{12}\otimes y_{01}
\\
x_{10}\otimes y_{10} & x_{10}\otimes y_{11} & x_{11}\otimes y_{10} & x_{11}\otimes y_{11} & x_{12}\otimes y_{10} & x_{12}\otimes y_{11}
\end{array} \right).
\]

Also, we define $X\oplus Y$ as the matrix of arrows of $\mathcal{F}$, schematically presented as
\[
\left(\begin{array}{cc} X & 0 \\ 0 & Y\end{array}\right).
\]
More precisely, for $X=(x_{ij})_{m\times n}$ and $Y=(y_{ij})_{q\times r}$ as above, $X\oplus Y$ is the \linebreak $(m+q)\times (n+r)$ matrix whose $ij$-entry is:
\begin{enumerate}
\item $x_{ij}$, when $i<m$, $j<n$,
\item $y_{(i-m)(j-n)}$, when $i\geq m$, $j\geq n$,
\item $0_{a^j,d^{i-m}}$, when $i\geq m$, $j<n$,
\item $0_{c^{j-n},b^i}$, when $i<m$, $j\geq n$.
\end{enumerate}

If $m=q$, $n=r$, for every $0\leq j< n$, $a^j=c^j$, and for every $0\leq i<m$, $b^i=d^i$, i.e.\ $X$ and $Y$ are of the same type having the corresponding elements in the same hom-sets, then $X+Y$ is the matrix of the same type whose $ij$-entry is $x_{ij}+y_{ij}$. If $n=q$ and for every $0\leq k<n$, $a^k=d^k$, i.e.\ $X$ is an $m\times n$ matrix, $Y$ is an $n\times r$ matrix and for every $0\leq i<m$, $0\leq j<r$, $0\leq k<n$ the composition $x_{ik}\circ y_{kj}$ is defined, then we define $X\circ Y$ as the $m\times r$ matrix whose $ij$-entry is $\sum_{k=0}^{n-1} x_{ik}\circ y_{kj}$ (this sum is defined since every $x_{ik}\circ y_{kj}$ is from $c^j$ to $b^i$).

Just by omitting the case (2) of \cite[Proposition~5.1]{PZ21} we obtain the following.
\begin{prop}\label{matrix1}
For $\bullet\in\{\otimes,\oplus,+,\circ\}$, we have
\[
M_{u_1\bullet u_2}=M_{u_1}\bullet M_{u_2}.
\]
\end{prop}

Our next proposition is related to \cite[Propositions~5.2, 8.2]{PZ21}
\begin{prop}\label{matrix2}
If $u$ is a primitive term of $\mathcal{F}$, then all the entries of the matrix $M_u$ are primitive terms of $\mathcal{F}$, not of the form $\pi$ or $\iota$, whose indices are $\oplus$-free.
\end{prop}
\begin{proof} We illustrate just a couple of cases. If $u$ is $\gamma$, for $\gamma\in\Gamma$, then $M_u$ is a $1\times 1$ matrix whose only entry is $\gamma$. The same holds when $\gamma$ is replaced by $\gamma^{-1}$. If $u$ is $\alpha^{-1}_{a,b,c}$, then for some $i_1, i_2, i_3$ and $j_1, j_2, j_3$
\begin{align*}
(M_u)_{i,j}&=\pi_{a\otimes (b\otimes c)}^i\circ \alpha^{-1}_{a,b,c}\circ \iota_{(a\otimes b)\otimes c}^j \\ &= (\pi_{a}^{i_1}\otimes (\pi_{b}^{i_2}\otimes \pi_{c}^{i_3})) \circ \alpha^{-1}_{a,b,c} \circ ((\iota_{a}^{j_1}\otimes \iota_{b}^{j_2})\otimes \iota_{c}^{j_3}) \\
&= \begin{cases}
\alpha^{-1}_{a^{i_1},b^{i_2}, c^{i_3}}, & i_1=j_1, \; i_2=j_2, \; i_3=j_3, \\
0_{(a^{j_1}\otimes b^{j_2})\otimes c^{j_3}, a^{i_1}\otimes (b^{i_2}\otimes c^{i_3})}, & \text{otherwise}.
\end{cases}
\end{align*}
If $u$ is $\sigma_{a,b}$, then for some $i_1, i_2$ and $j_1, j_2$
\begin{align*}
(M_u)_{i,j}&=\pi_{b\otimes a}^i\circ \sigma_{a,b}\circ \iota_{a\otimes b}^j \\ &= (\pi_{b}^{i_1}\otimes \pi_{a}^{i_2}) \circ \sigma_{a,b} \circ (\iota_{a}^{j_1}\otimes \iota_{b}^{j_2}) \\
&= \begin{cases}
\sigma_{a^{j_1},b^{j_2}}, & j_1=i_2, \; j_2=i_1, \\
0_{a^{j_1}\otimes b^{j_2}, b^{i_1}\otimes a^{i_2}}, & \text{otherwise}.
\end{cases}
\end{align*}
If $u$ is $\varepsilon_a$, then $M_u$ is a row matrix and for some $j_1$, $j_2$ we have
\begin{align*}
(M_u)_{1,j}&=\varepsilon_a\circ \iota_{a\otimes a^\ast}^j = \varepsilon_a \circ (\iota_{a}^{j_1}\otimes (\pi_{a}^{j_2})^\ast) \stackrel{(\ref{Jednakost:Diprirodnost})}{=} \varepsilon_{a^{j_2}} \circ ((\pi_{a}^{j_2} \circ\iota_{a}^{j_1})\otimes (a^{j_2})^\ast) \\
&= \begin{cases}
\varepsilon_{a^{j_2}}, & j_1=j_2, \\
0_{a^{j_1}\otimes (a^{j_2})^\ast, I}, & \text{otherwise}.
\end{cases}
\end{align*}
If $u$ is $\pi^1_{a,b}$, then
$(M_u)_{i,j}=\pi_a^i\circ \pi^1_{a,b} \circ \iota_{a\oplus b}^{j}$, which is either $\pi_a^i\circ \pi^1_{a,b}\circ \iota^1_{a,b}\circ \iota^{j_1}_a$ for some $j_1$, or
$\pi_a^i\circ \pi^1_{a,b}\circ \iota^2_{a,b}\circ \iota^{j_2}_b$, for some $j_2$. The statement holds since
\[
\pi_a^i\circ \pi^1_{a,b}\circ \iota^1_{a,b}\circ \iota^{j_1}_a=
\begin{cases}
\mathbf{1}_{a^i}, & i=j_1, \\
0_{a^{j_1}, a^i}, & \text{otherwise},
\end{cases}
\quad\quad
\pi_a^i\circ \pi^1_{a,b}\circ \iota^2_{a,b}\circ \iota^{j_2}_b=
0_{b^{j_2},a^i}.
\]
We proceed analogously when $u$ is $\mj_a$, $\alpha_{a,b,c}$, $\lambda_a$, $\lambda^{-1}_a$, $\eta_a$, $\pi^2_{a,b}$, $\iota^1_{a,b}$, $\iota^2_{a,b}$ or $0_{a,b}$.
\end{proof}

\begin{cor}\label{4.3}
For every arrow $u$ of $\mathcal{F}$, every entry of $M_u$ is expressible free of $\oplus$, $\iota$ and $\pi$.
\end{cor}

As a consequence of Remark~\ref{3.1} and Corollary~\ref{4.3} we have the following.

\begin{cor}\label{4.4}
Every arrow of $\mathcal{F}$ whose source and target are $\oplus$-free is expressible free of $\oplus$, $\iota$ and $\pi$.
\end{cor}

\section{The category $1\cobG^\oplus$ and coherence}\label{graficki}

The aim of this section is to introduce a category providing a diagrammatical checking of validity of quantum protocols. We start with a set $\Gamma$ (usually finite and non-empty) and a group $\mathfrak{G}$ freely generated by $\Gamma$. The category $1\cob$ and the group $\mathfrak{G}$ deliver the category $1\cobG$ through the following construction. The objects of $1\cobG$ are the objects of $1\cob$ and in order to define the arrows of $1\cobG$ we introduce the notions of $\mathfrak{G}$-components and $\mathfrak{G}$-cobordisms first.

A $\mathfrak{G}$-\emph{component} is a connected, oriented 1-manifold possibly with boundaries, together with an element of $\mathfrak{G}$. When a $\mathfrak{G}$-component is closed, we call it $\mathfrak{G}$-\emph{circle}, otherwise it is a $\mathfrak{G}$-\emph{segment}. We call the element of $\mathfrak{G}$ associated to a component the \emph{label} of this component.

A $\mathfrak{G}$-\emph{cobordism} from $a$ to $b$ is a finite collection of $\mathfrak{G}$-components whose underlying manifold is $M$, together with two embeddings $f_0\colon a\to M$ and $f_1 \colon b\to M$ such that $(M,f_0,f_1)$ is a 1-cobordism from $a$ to $b$. Two $\mathfrak{G}$-cobordisms are \emph{equivalent}, when the underlying 1-cobordisms are equivalent and the homeomorphism $F$ witnessing this equivalence satisfies:
\begin{enumerate}
\item every segment and its $F$-image are labeled by the same element of $\mathfrak{G}$;
\item the labels of a circle and its $F$-image could differ only in a circular permutation, i.e.\ if one is of the form $g_2\cdot g_1$, the other could be $g_1\cdot g_2$.
\end{enumerate}
The operation $\dagger$ on $\mathfrak{G}$-cobordisms is defined so that it is applied to the underlying cobordism and every label is replaced by its inverse.
\begin{center}
\begin{tikzpicture}[scale=0.45,line cap=round,line join=round,x=1.0cm,y=1.0cm]
\draw [line width=1pt][shift={(3.0,4.0)}] plot[domain=3.141592653589793:6.283185307179586,variable=\t]({1.0*2.0*cos(\t r)+-0.0*2.0*sin(\t r)},{0.0*2.0*cos(\t r)+1.0*2.0*sin(\t r)});
\draw [line width=1pt](1.0,0.0)-- (3.0,4.0);
\draw [line width=1pt](13.0,4.0)-- (15.0,0.0);
\draw [line width=1pt][shift={(15.0,0.0)}] plot[domain=0.0:3.141592653589793,variable=\t]({1.0*2.0*cos(\t r)+-0.0*2.0*sin(\t r)},{0.0*2.0*cos(\t r)+1.0*2.0*sin(\t r)});
\draw (0.5,5) node[anchor=north west] {$\ps$};
\draw (2.5,5) node[anchor=north west] {$\ps$};
\draw (0.5,0) node[anchor=north west] {$\ps$};
\draw (12.5,5) node[anchor=north west] {$\ps$};
\draw (12.5,0) node[anchor=north west] {$\ps$};
\draw (14.5,0) node[anchor=north west] {$\ps$};
\draw (4.5,5) node[anchor=north west] {$\ms$};
\draw (16.5,0) node[anchor=north west] {$\ms$};
\draw (1.4,1.6) node[anchor=north west] {$g$};
\draw (13.3,4.2) node[anchor=north west] {$g^{-1}$};
\draw (3.6,2.3) node[anchor=north west] {$h$};
\draw (15.8,2.8) node[anchor=north west] {$h^{-1}$};
\draw (-1,2.8) node[anchor=north west] {$f\colon$};
\draw (11,2.8) node[anchor=north west] {$f^\dagger\colon$};
\draw [->, >=stealth, line width=1.3pt] (1.0816257580341195,0.16325151606823912) -- (1.0,0.0);
\draw [->, >=stealth, line width=1.3pt] (4.99,3.8813279284517814) -- (5.0,4.0);
\draw [->, >=stealth, line width=1.3pt] (14.918608800547068,0.1627823989058661) -- (15.0,0.0);
\draw [->, >=stealth, line width=1.3pt] (13.01,0.1329158493488266) -- (13.0,0.0);
\draw [fill=black] (5.0,4.0) circle (3pt);
\draw [fill=black] (1.0,4.0) circle (3pt);
\draw [fill=black] (1.0,0.0) circle (3pt);
\draw [fill=black] (3.0,4.0) circle (3pt);
\draw [fill=black] (13.0,4.0) circle (3pt);
\draw [fill=black] (15.0,0.0) circle (3pt);
\draw [fill=black] (13.0,0.0) circle (3pt);
\draw [fill=black] (17.0,0.0) circle (3pt);
\end{tikzpicture}
\end{center}
The category $1\cobG$ has the equivalence classes of $\mathfrak{G}$-cobordisms as arrows. The identity $\mj_a\colon a\to a$  is the ordinary identity cobordism in which every segment is labeled by the \emph{neutral} $e$ of $\mathfrak{G}$. Two $\mathfrak{G}$-cobordisms are composed so that the underlying 1-cobordisms are composed in the ordinary manner. It remains to label the resulting segments and circles: if the segments $l_1,\ldots,l_k$ with labels $g_1,\ldots,g_k$ respectively, are glued together in a segment or a circle of the resulting 1-cobordism so that the terminal point of $l_i$ is identified with the initial point of $l_{i+1}$, then $g_k\cdot\ldots\cdot g_1$ is the (``a'' in the case of a circle) label of the resulting component. The category $1\cobG$ has dagger strict compact closed structure inherited from $1\cob$ (all segments in canonical arrows $\sigma$, $\eta$ and $\varepsilon$ are labeled in $1\cobG$ by the neutral $e$ of $\mathfrak{G}$).

Let us compare the above construction with the construction of the category $G\mathcal{A}$ given in \cite{KL80}, for $\mathcal{A}$ being the groupoid $\mathfrak{G}$, i.e.\ the category with a single object $p$ whose arrows are the elements of $\mathfrak{G}$ and the composition is the multiplication in $\mathfrak{G}$. The main theorem of \cite{KL80} claims that $G\mathfrak{G}$ is a compact closed category freely generated by the category $\mathfrak{G}$. This means that for every compact closed category $\mathcal{C}$ and a function $\varphi$ from the set $\Gamma$ to the set of automorphisms of an object $c$ of $\mathcal{C}$, there exists a unique functor
\[
F\colon G\mathfrak{G} \to \mathcal{C}
\]
that strictly preserves compact closed structure, and such that $Fp=c$ and for every $\gamma\in \Gamma$, $F\gamma=\varphi(\gamma)$.

One could easily conclude that $1\cobG$ is a strict compact closed version of $G\mathfrak{G}$. More precisely, the functor $F^{\mathrm{st}}\colon G\mathfrak{G}\to 1\cobG$ obtained by the above universal property of $G\mathfrak{G}$ is defined as follows. It maps every object $X$ of $G\mathfrak{G}$ to the sequence of signs corresponding to the signed set $P(X)$ (see \cite[Section~3]{KL80}). On arrows it is defined just by replacing the source and the target by the corresponding sequences of signs. Namely, an arrow of $G\mathfrak{G}$ (see \cite[Section~3]{KL80}) is represented by a triple, which is essentially contained in the notion of  $\mathfrak{G}$-cobordism. Hence, $F^{\mathrm{st}}$ maps an arrow (neglecting its source and target) to itself. It is straightforward to see that we have the following.
\begin{prop}\label{strictification}
The functor $F^{\mathrm{st}}\colon G\mathfrak{G}\to 1\cobG$ is faithful.
\end{prop}

In another words, to pass from $1\cobG$ to $G\mathfrak{G}$ one has to ``decorate'' the objects of $1\cobG$ with propositional formulae built in the language including single propositional letter, constant $I$, unary connective $\ast$ and binary connective $\otimes$. However, this just disguises strict compact closed nature of $1\cobG$, which is intrinsic to this category.

Let $1\cobG^+$ be the category with the same objects as $1\cobG$, while the arrows of $1\cobG^+$ from $a$ to $b$ are the formal sums of arrows of $1\cobG$ from $a$ to $b$. These formal sums may be represented by finite (possibly empty) multisets of $\mathfrak{G}$-cobordisms from $a$ to $b$. Formally, a multiset of elements of a set $X$  is a function from $X$ to the set of natural numbers (including zero). Less formally, it is a set in which elements may have multiple occurrences.

We abuse the notation by using the set brackets $\{$, $\}$ for multisets and by denoting a singleton multiset $\{f\}$ by $f$. Note that in this notation $\bigcirc +\bigcirc$, i.e.\ $\{\bigcirc,\bigcirc\}$ is not equal to $\bigcirc\bigcirc$, i.e.\ $\{\bigcirc\bigcirc\}$, where $\bigcirc$ is a circular component with arbitrary label.

The identity arrow $\mj_a\colon a\to a$ is the singleton multiset $\mj_a\colon a\to a$, while the composition of $\{f_j\colon a\to b\mid j\in J\}$ and $\{f_k\colon b\to c\mid k\in K\}$ is
\[
\{f_k\circ f_j\colon a\to c\mid j\in J, k\in K\}.
\]

Again, because of too many roles of $\emptyset$ in this paper, we denote the empty multiset of $\mathfrak{G}$-cobordisms from $a$ to $b$ by $0_{a,b}$, and call it \emph{zero-arrow}.
The existence of zero-arrows implies that every hom-set in $1\cobG^+$ is inhabited. The category $1\cobG^+$ is enriched over the category $\cmd$. The addition in ${\rm Hom}\,(a,b)$ is the operation + (disjoint union) on multisets and the neutral is $0_{a,b}$.

Let $1\cobG^\oplus$ be the biproduct completion of $1\cobG^+$ constructed as follows (see \cite[Section~5.1]{S07}). The objects of $1\cobG^\oplus$ are finite (possibly empty) sequences $( a_0,\ldots,a_{n-1})$, $n\geq 0$, of objects $a_0,\ldots,a_{n-1}$ of $1\cobG$. (We abuse the notation by denoting a singleton sequence $(a_0)$ by $a_0$.) For example, $(\ps\ps\ms\ps\ms\ms,o,\ps,\ms\ms\ps,o )$ is an object of $1\cobG^\oplus$. (Here, according to our convention, $o$ denotes the empty sequence of oriented points.)

The empty sequence of objects of $1\cobG$ plays the role of zero-object in $1\cobG^\oplus$, and for the above reasons we denote it by $\mathbf{0}$ and not by $\emptyset$. Note the distinction between this object and the object presented by the singleton sequence $o$ whose only member is the empty sequence of oriented points.

The arrows of $1\cobG^\oplus$ from $( a_0,\ldots,a_{n-1})$ to $( b_0,\ldots,b_{m-1})$ are the $m\times n$ matrices whose $ij$-entry is an arrow of $1\cobG^+$ from $a_j$ to $b_i$. If $m=0$, i.e.\ $( b_0,\ldots,b_{m-1})=\mathbf{0}$, then the empty matrix is the unique arrow from $a=( a_0,\ldots,a_{n-1})$ to $\mathbf{0}$, and we denote it by $0_{a,0}$. We proceed analogously when $n=0$.

The identity arrow $\mj_a$ on $a=( a_0,\ldots,a_{n-1})$ in $1\cobG^\oplus$ is the $n\times n$ matrix with corresponding identity arrows of $1\cobG^+$ in the main diagonal and corresponding zero-arrows of $1\cobG^+$ elsewhere. The arrows are composed by the rule of matrix multiplication, save that the addition and multiplication in a field are replaced by addition in hom-sets and composition in the category $1\cobG^+$. For $a=( a_0,\ldots,a_{n-1})$ and $b=( b_0,\ldots,b_{m-1})$, we denote by $0_{a,b}$, or simply $0_{m\times n}$, the $m\times n$ matrix whose $ij$-entry is the zero-arrow $0_{a_j,b_i}$ of $1\cobG^+$. In the limit cases, when we compose the empty matrices $0_{a,0}$ and $0_{0,b}$, we define the result as the \emph{zero-matrix} $0_{a,b}$.

\begin{prop}\label{dagger}
The category $1\cobG^\oplus$ has the structure of strict compact closed category with bi\-products. The group of automorphisms of the object $\ps$ in this category is isomorphic to $\mathfrak{G}$. Moreover, $\dagger$ is definable in $1\cobG^\oplus$, which makes it dagger strict compact closed category with dagger biproducts, while the automorphisms of $\ps$ are unitary.
\end{prop}
\begin{proof}
We define the compact closed structure on $1\cobG^\oplus$ as follows. The tensor product of objects $( a_0,\ldots,a_{n-1})$ and $( b_0,\ldots,b_{m-1})$ is the object $( a_0\otimes b_0,\ldots,a_0\otimes b_{m-1},\ldots,a_{n-1}\otimes b_{m-1})$ of $1\cobG^\oplus$. If either $n=0$ or $m=0$, the result is zero-object $\mathbf{0}$. The unit object is $o$. The tensor product of arrows of $1\cobG^\oplus$ is defined as the Kronecker product of matrices over a field, save that this time the multiplication in the field is replaced by the tensor product in the category $1\cobG^+$.

The arrows $\alpha$ and $\lambda$ are identities. For $a=( a_0,\ldots,a_{n-1})$ and $b=( b_0,\ldots,b_{m-1})$, the $(n\cdot m)\times (m\cdot n)$ matrix $\sigma_{a,b}$ (an arrow of $1\cobG^\oplus$) is defined as the permutation matrix representing the isomorphism between $V\otimes W$ and $W\otimes V$ for $V$ being $n$-dimensional and $W$ being $m$ dimensional vector space, save that instead of the entries $1$, we have arrows $\sigma$ from $1\cobG^+$, with corresponding indices, and instead of entries $0$, we have zero-arrows (i.e. empty multisets) of $1\cobG^+$ with corresponding indices. For example, if  $a=( a_0,a_1,a_2)$ and $b=( b_0,b_1)$, the matrix $\sigma_{a,b}$ (with indices of zero-arrows omitted) is
\[
\left(\begin{array}{cccccc} \sigma_{a_0,b_0} & 0_{a_0\otimes b_1,b_0\otimes a_0} & 0 & 0 & 0 & 0
\\ 0 & 0 & \sigma_{a_1,b_0} & 0 & 0 & 0
\\ 0 & 0 & 0 & 0 & \sigma_{a_2,b_0} & 0
\\ 0 & \sigma_{a_0,b_1} & 0 & 0 & 0 & 0
\\ 0 & 0 & 0 & \sigma_{a_1,b_1} & 0 & 0
\\ 0 & 0 & 0 & 0 & 0 & \sigma_{a_2,b_1}
\end{array}\right).
\]

The operation $\ast$ on objects of $1\cobG^\oplus$ is defined componentwise. The arrow $\eta_a$ for $a=( a_0,\ldots,a_{n-1})$ is the $n^2\times 1$ matrix with the singleton multiset $\eta_{a_k}$ in the $k\cdot (n+1)$-th row, for $0\leq k<n$, and zero-arrows of $1\cobG^+$, with corresponding indices, elsewhere. The arrow $\varepsilon_a$ is the $1\times n^2$ matrix having $\varepsilon_{a_k}$  in the $k\cdot (n+1)$-th column, for $0\leq k<n$, and zero-arrows of $1\cobG^+$, with corresponding indices, elsewhere. One can verify that the equalities~\ref{1}-\ref{18} hold in $1\cobG^\oplus$. Moreover, the arrows $u_{a,b}$, $v$ and $w_a$ defined in Section~\ref{closed} are identities. Hence $1\cobG^\oplus$ is a strict compact closed category.

The operation $+$ on arrows from $( a_0,\ldots,a_{n-1})$ to $( b_0,\ldots,b_{m-1})$ is defined componentwise and zero-matrices are the neutrals for this operation. The equations \ref{16}-\ref{18} hold, which guarantees that $1\cobG^\oplus$ is enriched over $\cmd$.

For objects $a=( a_0,\ldots,a_{n-1})$ and $b=( b_0,\ldots,b_{m-1})$ the object $a\oplus b$ is the sequence
\[
( a_0,\ldots,a_{n-1}, b_0,\ldots,b_{m-1}).
\]
The object $\mathbf{0}$ is the zero-object of $1\cobG^\oplus$ and it is the neutral for $\oplus$. For arrows $A_{m\times n}$ and $B_{p\times q}$ of $1\cobG^\oplus$ its \emph{direct sum} $A\oplus B$ is the $(m+p)\times (n+q)$ matrix
\[
\left(\begin{array}{cc} A & 0_{m\times q} \\ 0_{p\times n} & B\end{array}\right).
\]
For $a=( a_0,\ldots,a_{n-1})$ and $b=( b_0,\ldots,b_{m-1})$, the arrows $\pi^1_{a,b}$, $\pi^2_{a,b}$, $\iota^1_{a,b}$ and $\iota^2_{a,b}$ are defined as
\[
\pi^1_{a,b}=\left(\begin{array}{cc} \mj_a & 0_{n\times m} \end{array}\right),\quad \pi^2_{a,b}=\left(\begin{array}{cc} 0_{m\times n}  & \mj_b \end{array}\right),
\]
\[
\iota^1_{a,b}=\left(\begin{array}{c} \mj_a \\ 0_{m\times n} \end{array}\right),\quad \iota^2_{a,b}=\left(\begin{array}{c} 0_{n\times m} \\ \mj_b \end{array}\right).
\]
After checking that the equalities \ref{3}-\ref{22} hold in $1\cobG^\oplus$, one concludes that this category is strict compact closed with biproducts.

That the group of automorphisms of the object $\ps$ in $1\cobG^\oplus$ is isomorphic to $\mathfrak{G}$ is shown as follows. Every arrow from $\ps$ to itself is a $1\times 1$ matrix whose entry is a multiset of arrows of $1\cobG$ from the singleton sequence of oriented points $\ps$ to itself. This multiset is a singleton in the case of an isomorphism, which follows from the fact that the composition in $1\cobG^+$ of a multiset of cardinality $n$ with a multiset of cardinality $m$ is a multiset of cardinality $n\cdot m$, and an isomorphism must be canceled to $\mj_{\ps}\colon \ps\to \ps$, which is the singleton multiset $\mj_{\ps}$. Hence, every isomorphism from $\ps$ to $\ps$ in $1\cobG^\oplus$ is of the form $\psi\colon \ps\to \ps$, for $\psi$ an arrow of $1\cobG$. Moreover, $\psi$ must be an isomorphism in $1\cobG$. An arrow of $1\cobG$ from $\ps$ to $\ps$ consists of a single $\mathfrak{G}$-segment and several (possibly zero) $\mathfrak{G}$-circles. Since $\psi$ is an isomorphism and $\mathfrak{G}$-circles are not cancelable, there are no $\mathfrak{G}$-circles in $\psi$ and it could be identified with the underlying $\mathfrak{G}$-segment. The label of this segment is the element of $\mathfrak{G}$ corresponding to the initial isomorphism of $1\cobG^\oplus$. It is evident that this correspondence is a one-to-one homomorphism.

The operation $\dagger$ on arrows of $1\cobG^+$ is defined as
\[
(\{f_j\colon a\to b\mid j\in J\})^\dagger=\{f_j^\dagger\colon b\to a\mid j\in J\}.
\]
The operation $\dagger$ on a matrix representing an arrow of $1\cobG^\oplus$ is defined by transposing this matrix, and by applying the operation $\dagger$, defined above, to each of its entries. In order to verify that $1\cobG^\oplus$ is dagger strict compact closed with dagger bi\-products, it remains to check that the equalities~\ref{4}-\ref{31} of Appendix~\ref{language} hold. The definition of $\dagger$ in $1\cobG$ guarantees that the automorphisms of $\ps$ are unitary.
\end{proof}

\begin{rem}
For our purposes it is useful to have a direct presentation of $\ulcorner f \urcorner$, $\llcorner f \lrcorner$, $f_\ast$ and $\langle f_1,\ldots,f_n\rangle$ at least for arrows $f,f_1,\ldots,f_n$ of $1\cobG$. The first three operations are defined as in $1\cob$ (the labels of $\mathfrak{G}$-components remain the same in the first two cases, while in the case of $f_\ast$ the labels become the inverses of the initial labels). The last operation (see the definition of biproducts in Section~\ref{closed}) produces the $n\times 1$ matrix
\[
\left(\begin{array}{c} f_1
\\
\vdots
\\
f_n
\end{array}\right).
\]
\end{rem}

\begin{rem} \label{Remark:Distributivity}
By relying on the equalities \ref{DistUpsilon} it is not difficult to show that the left distributivity isomorphism $\upsilon_{a,b,c}\colon (a\oplus b)\otimes c\to (a\otimes c)\oplus (b\otimes c)$ is the identity in the category $1\cobG^\oplus$. Similarly, if we assume that $a$ is a singleton sequence, then by relying the equalities \ref{DistTau} we can show that the right distributivity isomorphism $\tau_{a,b,c}\colon a\otimes (b\oplus c)\to (a\otimes b)\oplus (a\otimes c)$ is also the identity in $1\cobG^\oplus$.
\end{rem}

\begin{rem}\label{functor F}
By the universal property of the category $\mathcal{F}$ from Section~\ref{free category}, there exists a unique functor $H\colon\mathcal{F}\to 1\cobG^\oplus$ that strictly preserves the compact closed structure with biproducts, for which $Hp=\ps$ and for every $\gamma\in \Gamma$, $H\gamma$ is the $\mathfrak{G}$-cobordism from $\ps$ to $\ps$ given by one $\mathfrak{G}$-segment labeled by $\gamma$. The isomorphism of $\mathcal{F}$ and $\mathcal{F}^\dagger$ from the proof of Proposition~\ref{free-iso} enables one to consider $H$ as a functor from $\mathcal{F}^\dagger$ to $1\cobG^\oplus$ that strictly preserves the dagger compact closed structure with dagger biproducts.
\end{rem}

\begin{prop}\label{coherence}
The functor $H\colon\mathcal{F}\to 1\cobG^\oplus$ is faithful.
\end{prop}
\begin{proof}
Let $f,g\colon a\to b$ be two arrows of $\mathcal{F}$ such that $Hf=Hg$, and let ${\rm I}_a= (\iota^0_a, \ldots,\iota^{n-1}_a)$ and $\Pi_b=(\pi^0_b,\ldots,\pi^{m-1}_b)$. By Corollary~\ref{3.3} and properties of biproducts it suffices to show that, for every $0\leq i<m$ and $0\leq j<n$,
\begin{equation}\label{atom}
\pi^i_b\circ f\circ\iota^j_a= \pi^i_b\circ g\circ\iota^j_a.
\end{equation}
By Corollary~\ref{4.4}, both sides of \ref{atom} are expressible free of $\oplus$, $\iota$ and $\pi$. By relying on the equalities \ref{17}, \ref{18}, \ref{25} and \ref{27}, both sides are expressible as sums of terms, which are all free of $\oplus$, $+$ and $0,\iota,\pi$-arrows. Here, the empty sum is denoted by $0_{a^j,b^i}$.

If one side of the above equality is equal to $0_{a^j,b^i}$, then it is mapped by $H$ to the empty multiset. By functorial properties of $H$, the sum at the other side must be mapped by $H$ to the empty multiset too, which means that this sum is empty, i.e.\ it is $0_{a^j,b^i}$.

It remains the case when for $n,m\geq 1$ the left-hand side of \ref{atom} is equal to $\sum_{k=1}^{n} f_k$ and the right-hand side of this equality is equal to $\sum_{k=1}^{m} g_k$, for $f_k,g_k$ free of $\oplus$, $+$ and $0,\iota,\pi$-arrows. We have that
\[
\{Hf_k\mid 1\leq k\leq n\}=H\sum_{k=1}^n f_k=H\sum_{k=1}^m g_k= \{Hg_k\mid 1\leq k\leq m\},
\]
which means that $n=m$, and modulo some permutation of elements of these multisets, for every $1\leq k\leq n$, $Hf_k=Hg_k$. The terms $f_k$ and $g_k$ belong entirely to the compact closed fragment generated by $\mathfrak{G}$. Hence, these terms represent arrows of a compact closed category $F\mathfrak{G}$ freely generated by $\mathfrak{G}$ (see \cite[Section~4]{KL80}). The functor $H$ restricts to $F\mathfrak{G}$ as the composition of an isomorphism (from $F\mathcal{A}$ to $G\mathcal{A}$, for $\mathcal{A}$ being $\mathfrak{G}$; see \cite[Sections~3-4]{KL80}) and the faithful functor $F^{\mathrm{st}}$ of Proposition~\ref{strictification}, which means that this restriction is faithful. We conclude that $f_k$ and $g_k$ represent the same arrow of $F\mathfrak{G}$, and hence of $\mathcal{F}$.
\end{proof}

\section{Validity of categorical quantum protocols}\label{validity}
It was suggested in \cite{AC04} that compact closed categories with biproducts provide a generalisation of von Neumann's presentation of quantum mechanics in terms of Hilbert spaces, \cite{vN32}. Such an approach is called \emph{categorical quantum mechanics}. For a survey of theory of categorical quantum mechanics, we recommend \cite{AC08,V12} and references therein.

In this section we use Proposition~\ref{coherence} to establish commutativity of diagrams in the category $\mathcal{F}$, which provides a verification of the corresponding protocols from the realm of categorical quantum mechanics. All the protocols verified in \cite{AC04} require a compact closed category with dagger biproducts, possessing some additional structure. For the first two protocols below, this extra structure consists of an object $Q$ (the \emph{qubit}), an arrow from $4\cdot I$ to $Q^\ast\otimes Q$ and a scalar $s$ satisfying some conditions listed in \cite[Section~9]{AC04}. (Here we abbreviate $((I\oplus I)\oplus I)\oplus I$ by $4\cdot I$, and more generally, $n\cdot a$ and $n\cdot f$ abbreviate the $n$-fold biproducts, associated to the left, of an object $a$ and an arrow $f$ respectively.)

However, the only additional structure upon a compact closed structure with dagger biproducts important for the verification diagrams consists of four unitary isomorphisms $\beta_1, \beta_2, \beta_3, \beta_4\colon Q\to Q$. Hence, to establish that the verification diagrams are commutative in an arbitrary such category, i.e.\ that the categorical quantum protocols are correct, it suffices to establish their commutativity in the compact closed category $\mathcal{F}$ with biproducts freely generated by the free group $\mathfrak{G}$ on four generators. (Since $\beta_1$ is standardly taken to be identity, a group with three generators suffices.) The role of the generator $p$ for objects of $\mathcal{F}$ (see Section~\ref{free category}) belongs now to the qubit $Q$.

Our Proposition~\ref{coherence} enables one to check the commutativity of diagrams in $\mathcal{F}$ by ``drawing pictures'' and this is the style of verification given below. The qubit $Q$ is interpreted in $1\cobG^\oplus$ as $\ps$. At some points we have to draw matrices of pictures and this is done in the first example below, otherwise just the $ij$ element of such a matrix is described. 

\subsection{Quantum teleportation}\label{teleportation}
Quantum teleportation is a well-known quantum protocol \cite{AC04,NC10}. Assume that Alice has a qubit in some state $\ket{\psi}$, and wishes to sent this state to Bob, without any knowledge of what this state is. This is done by taking an entangled pair of qubits (EPR-pair, $\ket{\beta_1}$) and sending one to Alice and another to Bob. Then, Alice measures (in the Bell basis) her qubit and the qubit that is entangle with the one Bob has. In the next step, she communicates the result of the measurement to Bob, who applies unitary corrections to his qubit, depending on the Alice's outcome. The final result is that Bob's qubit is in the same state as Alice's qubit originally was (Alice does not have a qubit in state $\ket{\psi}$ after this protocol is done).

\begin{equation} \label{Diagram:QT}
\begin{tikzcd}[row sep=1cm, every label/.append style = {font = \small}]
Q_a \arrow[dddddd, "\Delta^4_{ac}"', bend right=95] \arrow[d, "\sigma_{I,a}\circ\lambda_a^{-1}"', "\text{ import unknown state}"] \\
Q_a\otimes I \arrow[d, "\mj_a\otimes\displaystyle\ulcorner\mj_{bc}\urcorner"', "\text{ produce EPR-pair}"]\\
Q_a\otimes (Q_b^\ast\otimes Q_c) \arrow[d, "\alpha_{a,b,c}"', "\text{ spatial delocation}"] \\
(Q_a\otimes Q_b^\ast)\otimes Q_c \arrow[d, "\displaystyle \langle\llcorner \beta_i^{ab} \lrcorner\rangle_{i=1}^{i=4}\otimes\mj_c"', "\text{ teleportation observation}"] \\
(4\cdot I)\otimes Q_c \arrow[d, "(4\cdot\lambda_c)\circ \upsilon_{4I,c}"', "\text{ classical comunication}"] \\
4\cdot Q_c \arrow[d, "\bigoplus_{i=1}^{i=4}(\beta_i^c)^{-1}"', "\text{ unitary correction}"] \\
4\cdot Q_c
\end{tikzcd}
\end{equation}

The correctness of the quantum teleportation protocol is expressed by commutativity of the diagram given in \cite[Theorem 9.1]{AC04}. One can easily factor the scalars out of both legs in this diagram just by appealing to the equalities~\ref{41}-\ref{43}. This makes the commutativity of Diagram~\ref{Diagram:QT} sufficient for the correctness of the protocol. We follow the terminology and notation introduced in \cite{AC04} in this diagram.

Note that we treat $Q_a$, $Q_b$ and $Q_c$ as three instances of the same object $Q$ of $\mathcal{F}$. Also, $\Delta^4_{a,c}$ is an abbreviation for $\langle \mj_Q,\mj_Q,\mj_Q,\mj_Q\rangle$. For example, producing the EPR-pair, means to apply the arrow $\mj_Q\otimes \ulcorner \mj_Q \urcorner$, which is interpreted in $1\cobG^\oplus$ as:
\begin{center}
\begin{tikzpicture}[scale=0.6,line cap=round,line join=round,x=1.0cm,y=1.0cm]
\draw [line width=1pt][shift={(4.0,0.0)}] plot[domain=0.0:3.141592653589793,variable=\t]({1.0*1.0*cos(\t r)+-0.0*1.0*sin(\t r)},{0.0*1.0*cos(\t r)+1.0*1.0*sin(\t r)});
\draw [line width=1pt](1.0,2.0)-- (1.0,0.0);
\draw (0.63,2.7) node[anchor=north west] {$\ps$};
\draw (0.63,0.1) node[anchor=north west] {$\ps$};
\draw (4.63,0.1) node[anchor=north west] {$\ps$};
\draw (2.63,0.1) node[anchor=north west] {$\ms$};
\draw (0.35,1.35) node[anchor=north west] {$e$};
\draw (3.63,1.6) node[anchor=north west] {$e$};
\draw [->,>=stealth,line width=1.3pt] (1.0,0.1196595073194644) -- (1.0,0.0);
\draw [->,>=stealth,line width=1.3pt] (4.992,0.05159051929475449) -- (5.0,0.0);
\draw [fill=black] (3.0,0.0) circle (2pt);
\draw [fill=black] (5.0,0.0) circle (2pt);
\draw [fill=black] (1.0,2.0) circle (2pt);
\draw [fill=black] (1.0,0.0) circle (2pt);
\end{tikzpicture}
\end{center}
This is the first nontrivial step in the diagram \ref{Diagram:QT}. (Note that since $1\cobG^\oplus$ is a strict compact closed category, the steps called ``import unknown state'' and ``spatial delocation'' are interpreted as identities in this category.)

In drawings of $\mathfrak{G}$-cobordisms, when we interpret the arrows of the diagram \ref{Diagram:QT} and the diagrams below, the orientation and the label $e$ (denoting the neutral of $\mathfrak{G}$) will be omitted. As we noted at the beginning of this section, our group $\mathfrak{G}$ is generated by the set $\Gamma=\{\beta_1, \beta_2, \beta_3, \beta_4\}$. The second nontrivial step in the diagram \ref{Diagram:QT} is the teleportation observation, given by $\langle\llcorner \beta_i \lrcorner\rangle_{i=1}^{i=4}\otimes \mj_Q$, or in terms of arrows of $1\cobG^\oplus$:
\vspace{-7mm}
\begin{center}
\[ \left(
\begin{tikzpicture}[scale=0.37,line cap=round,line join=round,x=1.0cm,y=1.0cm,baseline={([yshift=-\the\dimexpr\fontdimen22\textfont2\relax] current bounding box.center)}]
\draw [line width=1pt][shift={(1.0,2.0)}] plot[domain=3.141592653589793:6.283185307179586,variable=\t]({1.0*1.0*cos(\t r)+-0.0*1.0*sin(\t r)},{0.0*1.0*cos(\t r)+1.0*1.0*sin(\t r)});
\draw [line width=1pt][shift={(9.0,2.0)}] plot[domain=3.141592653589793:6.283185307179586,variable=\t]({1.0*1.0*cos(\t r)+-0.0*1.0*sin(\t r)},{0.0*1.0*cos(\t r)+1.0*1.0*sin(\t r)});
\draw [line width=1pt](12.0,2.0)-- (12.0,0.0);
\draw [line width=1pt](4.0,2.0)-- (4.0,0.0);
\draw [line width=1pt][shift={(17.0,2.0)}] plot[domain=3.141592653589793:6.283185307179586,variable=\t]({1.0*1.0*cos(\t r)+-0.0*1.0*sin(\t r)},{0.0*1.0*cos(\t r)+1.0*1.0*sin(\t r)});
\draw [line width=1pt](20.0,2.0)-- (20.0,0.0);
\draw [line width=1pt][shift={(25.0,2.0)}] plot[domain=3.141592653589793:6.283185307179586,variable=\t]({1.0*1.0*cos(\t r)+-0.0*1.0*sin(\t r)},{0.0*1.0*cos(\t r)+1.0*1.0*sin(\t r)});
\draw [line width=1pt](28.0,2.0)-- (28.0,0.0);
\draw (-0.6,3.1) node[anchor=north west] {$\ps$};
\draw (3.4,3.1) node[anchor=north west] {$\ps$};
\draw (3.4,0) node[anchor=north west] {$\ps$};
\draw (7.4,3.1) node[anchor=north west] {$\ps$};
\draw (11.4,3.1) node[anchor=north west] {$\ps$};
\draw (11.4,0) node[anchor=north west] {$\ps$};
\draw (15.4,3.1) node[anchor=north west] {$\ps$};
\draw (19.4,3.1) node[anchor=north west] {$\ps$};
\draw (19.4,0) node[anchor=north west] {$\ps$};
\draw (23.4,3.1) node[anchor=north west] {$\ps$};
\draw (27.4,3.1) node[anchor=north west] {$\ps$};
\draw (27.4,0) node[anchor=north west] {$\ps$};
\draw (1.4,3.1) node[anchor=north west] {$\ms$};
\draw (9.4,3.1) node[anchor=north west] {$\ms$};
\draw (17.4,3.1) node[anchor=north west] {$\ms$};
\draw (25.4,3.1) node[anchor=north west] {$\ms$};
\draw (0.4,1.2) node[anchor=north west] {$\beta_1$};
\draw (8.4,1.2) node[anchor=north west] {$\beta_2$};
\draw (16.4,1.2) node[anchor=north west] {$\beta_3$};
\draw (24.4,1.2) node[anchor=north west] {$\beta_4$};
\draw [fill=black] (2.0,2.0) circle (3pt);
\draw [fill=black] (0.0,2.0) circle (3pt);
\draw [fill=black] (10.0,2.0) circle (3pt);
\draw [fill=black] (8.0,2.0) circle (3pt);
\draw [fill=black] (12.0,2.0) circle (3pt);
\draw [fill=black] (12.0,0.0) circle (3pt);
\draw [fill=black] (4.0,2.0) circle (3pt);
\draw [fill=black] (4.0,0.0) circle (3pt);
\draw [fill=black] (18.0,2.0) circle (3pt);
\draw [fill=black] (16.0,2.0) circle (3pt);
\draw [fill=black] (20.0,2.0) circle (3pt);
\draw [fill=black] (20.0,0.0) circle (3pt);
\draw [fill=black] (26.0,2.0) circle (3pt);
\draw [fill=black] (24.0,2.0) circle (3pt);
\draw [fill=black] (28.0,2.0) circle (3pt);
\draw [fill=black] (28.0,0.0) circle (3pt);
\end{tikzpicture}
\right)^T \]
\end{center}

\noindent By composing this $4\times 1$ matrix with the $1\times 1$ matrix representing production of EPR-pair, we get
\vspace{-7mm}
\begin{center}
\[ \left(
\begin{tikzpicture}[scale=0.37,line cap=round,line join=round,x=1.0cm,y=1.0cm,baseline={([yshift=-\the\dimexpr\fontdimen22\textfont2\relax] current bounding box.center)}]
\draw [line width=1pt][shift={(1.0,2.0)}] plot[domain=3.141592653589793:6.283185307179586,variable=\t]({1.0*1.0*cos(\t r)+-0.0*1.0*sin(\t r)},{0.0*1.0*cos(\t r)+1.0*1.0*sin(\t r)});
\draw [line width=1pt][shift={(9.0,2.0)}] plot[domain=3.141592653589793:6.283185307179586,variable=\t]({1.0*1.0*cos(\t r)+-0.0*1.0*sin(\t r)},{0.0*1.0*cos(\t r)+1.0*1.0*sin(\t r)});
\draw [line width=1pt](12.0,2.0)-- (12.0,0.0);
\draw [line width=1pt](4.0,2.0)-- (4.0,0.0);
\draw [line width=1pt][shift={(17.0,2.0)}] plot[domain=3.141592653589793:6.283185307179586,variable=\t]({1.0*1.0*cos(\t r)+-0.0*1.0*sin(\t r)},{0.0*1.0*cos(\t r)+1.0*1.0*sin(\t r)});
\draw [line width=1pt](20.0,2.0)-- (20.0,0.0);
\draw [line width=1pt][shift={(25.0,2.0)}] plot[domain=3.141592653589793:6.283185307179586,variable=\t]({1.0*1.0*cos(\t r)+-0.0*1.0*sin(\t r)},{0.0*1.0*cos(\t r)+1.0*1.0*sin(\t r)});
\draw [line width=1pt](28.0,2.0)-- (28.0,0.0);
\draw (-1.1,2.6) node[anchor=north west] {$\ps$};
\draw (2.9,2.6) node[anchor=north west] {$\ps$};
\draw (3.4,0.2) node[anchor=north west] {$\ps$};
\draw (6.9,2.6) node[anchor=north west] {$\ps$};
\draw (10.9,2.6) node[anchor=north west] {$\ps$};
\draw (11.4,0.2) node[anchor=north west] {$\ps$};
\draw (14.9,2.6) node[anchor=north west] {$\ps$};
\draw (18.9,2.6) node[anchor=north west] {$\ps$};
\draw (19.4,0.2) node[anchor=north west] {$\ps$};
\draw (22.9,2.6) node[anchor=north west] {$\ps$};
\draw (26.9,2.6) node[anchor=north west] {$\ps$};
\draw (27.4,0.2) node[anchor=north west] {$\ps$};
\draw (0.9,2.6) node[anchor=north west] {$\ms$};
\draw (8.9,2.6) node[anchor=north west] {$\ms$};
\draw (16.9,2.6) node[anchor=north west] {$\ms$};
\draw (24.9,2.6) node[anchor=north west] {$\ms$};
\draw (0.4,1.2) node[anchor=north west] {$\beta_1$};
\draw (8.4,1.2) node[anchor=north west] {$\beta_2$};
\draw (16.4,1.2) node[anchor=north west] {$\beta_3$};
\draw (24.4,1.2) node[anchor=north west] {$\beta_4$};
\draw [line width=1pt](0.0,4.0)-- (0.0,2.0);
\draw [line width=1pt](8.0,4.0)-- (8.0,2.0);
\draw [line width=1pt](16.0,4.0)-- (16.0,2.0);
\draw [line width=1pt](24.0,4.0)-- (24.0,2.0);
\draw [line width=1pt][shift={(3.0,2.0)}] plot[domain=0.0:3.141592653589793,variable=\t]({1.0*1.0*cos(\t r)+-0.0*1.0*sin(\t r)},{0.0*1.0*cos(\t r)+1.0*1.0*sin(\t r)});
\draw [line width=1pt][shift={(11.0,2.0)}] plot[domain=0.0:3.141592653589793,variable=\t]({1.0*1.0*cos(\t r)+-0.0*1.0*sin(\t r)},{0.0*1.0*cos(\t r)+1.0*1.0*sin(\t r)});
\draw [line width=1pt][shift={(19.0,2.0)}] plot[domain=0.0:3.141592653589793,variable=\t]({1.0*1.0*cos(\t r)+-0.0*1.0*sin(\t r)},{0.0*1.0*cos(\t r)+1.0*1.0*sin(\t r)});
\draw [line width=1pt][shift={(27.0,2.0)}] plot[domain=0.0:3.141592653589793,variable=\t]({1.0*1.0*cos(\t r)+-0.0*1.0*sin(\t r)},{0.0*1.0*cos(\t r)+1.0*1.0*sin(\t r)});
\draw (-0.6,5.2) node[anchor=north west] {$\ps$};
\draw (7.4,5.2) node[anchor=north west] {$\ps$};
\draw (15.4,5.2) node[anchor=north west] {$\ps$};
\draw (23.4,5.2) node[anchor=north west] {$\ps$};
\draw [fill=black] (2.0,2.0) circle (3pt);
\draw [fill=black] (0.0,2.0) circle (3pt);
\draw [fill=black] (10.0,2.0) circle (3pt);
\draw [fill=black] (8.0,2.0) circle (3pt);
\draw [fill=black] (12.0,2.0) circle (3pt);
\draw [fill=black] (12.0,0.0) circle (3pt);
\draw [fill=black] (4.0,2.0) circle (3pt);
\draw [fill=black] (4.0,0.0) circle (3pt);
\draw [fill=black] (18.0,2.0) circle (3pt);
\draw [fill=black] (16.0,2.0) circle (3pt);
\draw [fill=black] (20.0,2.0) circle (3pt);
\draw [fill=black] (20.0,0.0) circle (3pt);
\draw [fill=black] (26.0,2.0) circle (3pt);
\draw [fill=black] (24.0,2.0) circle (3pt);
\draw [fill=black] (28.0,2.0) circle (3pt);
\draw [fill=black] (28.0,0.0) circle (3pt);
\draw [fill=black] (0.0,4.0) circle (3pt);
\draw [fill=black] (8.0,4.0) circle (3pt);
\draw [fill=black] (16.0,4.0) circle (3pt);
\draw [fill=black] (24.0,4.0) circle (3pt);
\end{tikzpicture}
\right)^T \]
\end{center}

\noindent At this point, Alice had performed her measurement, and communicated the result to Bob using classical interchange of bits. By Remark \ref{Remark:Distributivity}, we know that the distributivity isomorphism $\upsilon_{4I,c}$ is the identity in $1\cobG^{\oplus}$. This, together with the strictness of this category, makes the step named ``classical communication'' trivial, i.e.\ it is interpreted as identity.

Next, Bob applies unitary corrections, given by $\oplus _{i=1}^{i=4}(\beta_i)^{-1}$. In our matrix representation, $\oplus$ correspond to the direct sum of matrices. We therefore have the unitary correction
\vspace{-7mm}
\begin{center}
\[ \left(  \begin{tikzpicture}[scale=0.7,baseline={([yshift=-\the\dimexpr\fontdimen22\textfont2\relax] current
bounding box.center)}][line cap=round,line join=round,x=1.0cm,y=1.0cm]
\draw [line width=1pt] (1,8.6)-- (1,7.4);
\draw [line width=1pt] (3,6.6)-- (3,5.4);
\draw [line width=1pt] (5,4.6)-- (5,3.4);
\draw [line width=1pt] (7,2.6)-- (7,1.4);
\draw (-0.1,8.5) node[anchor=north west] {$\beta_1^{-1}$};
\draw (1.9,6.5) node[anchor=north west] {$\beta_2^{-1}$};
\draw (3.9,4.5) node[anchor=north west] {$\beta_3^{-1}$};
\draw (5.9,2.5) node[anchor=north west] {$\beta_4^{-1}$};
\draw (0.68,9.2) node[anchor=north west] {$\ps$};
\draw (0.68,7.4) node[anchor=north west] {$\ps$};
\draw (2.68,7.2) node[anchor=north west] {$\ps$};
\draw (2.68,5.4) node[anchor=north west] {$\ps$};
\draw (4.68,5.2) node[anchor=north west] {$\ps$};
\draw (4.68,3.4) node[anchor=north west] {$\ps$};
\draw (6.68,3.2) node[anchor=north west] {$\ps$};
\draw (6.68,1.4) node[anchor=north west] {$\ps$};
\draw (2.68,8.4) node[anchor=north west] {$0$};
\draw (4.68,8.4) node[anchor=north west] {$0$};
\draw (6.68,8.4) node[anchor=north west] {$0$};
\draw (4.68,6.4) node[anchor=north west] {$0$};
\draw (6.68,6.4) node[anchor=north west] {$0$};
\draw (6.68,4.4) node[anchor=north west] {$0$};
\draw (0.68,6.4) node[anchor=north west] {$0$};
\draw (0.68,4.4) node[anchor=north west] {$0$};
\draw (0.68,2.4) node[anchor=north west] {$0$};
\draw (2.68,4.4) node[anchor=north west] {$0$};
\draw (2.68,2.4) node[anchor=north west] {$0$};
\draw (4.68,2.4) node[anchor=north west] {$0$};
\begin{scriptsize}
\draw [fill=black] (1,8.6) circle (1.5pt);
\draw [fill=black] (1,7.4) circle (1.5pt);
\draw [fill=black] (3,6.6) circle (1.5pt);
\draw [fill=black] (3,5.4) circle (1.5pt);
\draw [fill=black] (5,4.6) circle (1.5pt);
\draw [fill=black] (5,3.4) circle (1.5pt);
\draw [fill=black] (7,2.6) circle (1.5pt);
\draw [fill=black] (7,1.4) circle (1.5pt);
\end{scriptsize}
\end{tikzpicture}\right) \]
\end{center}

\noindent By composing the last two matrices, we get the final result
\vspace{-7mm}
\begin{center}
\[ \left(
\begin{tikzpicture}[scale=0.37,line cap=round,line join=round,x=1.0cm,y=1.0cm,baseline={([yshift=-\the\dimexpr\fontdimen22\textfont2\relax] current bounding box.center)}]
\draw [line width=1pt][shift={(1.0,2.0)}] plot[domain=3.141592653589793:6.283185307179586,variable=\t]({1.0*1.0*cos(\t r)+-0.0*1.0*sin(\t r)},{0.0*1.0*cos(\t r)+1.0*1.0*sin(\t r)});
\draw [line width=1pt][shift={(9.0,2.0)}] plot[domain=3.141592653589793:6.283185307179586,variable=\t]({1.0*1.0*cos(\t r)+-0.0*1.0*sin(\t r)},{0.0*1.0*cos(\t r)+1.0*1.0*sin(\t r)});
\draw [line width=1pt](12.0,2.0)-- (12.0,0.0);
\draw [line width=1pt](4.0,2.0)-- (4.0,0.0);
\draw [line width=1pt][shift={(17.0,2.0)}] plot[domain=3.141592653589793:6.283185307179586,variable=\t]({1.0*1.0*cos(\t r)+-0.0*1.0*sin(\t r)},{0.0*1.0*cos(\t r)+1.0*1.0*sin(\t r)});
\draw [line width=1pt](20.0,2.0)-- (20.0,0.0);
\draw [line width=1pt][shift={(25.0,2.0)}] plot[domain=3.141592653589793:6.283185307179586,variable=\t]({1.0*1.0*cos(\t r)+-0.0*1.0*sin(\t r)},{0.0*1.0*cos(\t r)+1.0*1.0*sin(\t r)});
\draw [line width=1pt](28.0,2.0)-- (28.0,0.0);
\draw (-1.1,2.6) node[anchor=north west] {$\ps$};
\draw (2.9,2.6) node[anchor=north west] {$\ps$};
\draw (2.9,0.7) node[anchor=north west] {$\ps$};
\draw (6.9,2.6) node[anchor=north west] {$\ps$};
\draw (10.9,2.6) node[anchor=north west] {$\ps$};
\draw (10.9,0.7) node[anchor=north west] {$\ps$};
\draw (14.9,2.6) node[anchor=north west] {$\ps$};
\draw (18.9,2.6) node[anchor=north west] {$\ps$};
\draw (18.9,0.7) node[anchor=north west] {$\ps$};
\draw (22.9,2.6) node[anchor=north west] {$\ps$};
\draw (26.9,2.6) node[anchor=north west] {$\ps$};
\draw (26.9,0.7) node[anchor=north west] {$\ps$};
\draw (0.9,2.6) node[anchor=north west] {$\ms$};
\draw (8.9,2.6) node[anchor=north west] {$\ms$};
\draw (16.9,2.6) node[anchor=north west] {$\ms$};
\draw (24.9,2.6) node[anchor=north west] {$\ms$};
\draw (0.4,1.2) node[anchor=north west] {$\beta_1$};
\draw (8.4,1.2) node[anchor=north west] {$\beta_2$};
\draw (16.4,1.2) node[anchor=north west] {$\beta_3$};
\draw (24.4,1.2) node[anchor=north west] {$\beta_4$};
\draw [line width=1pt](0.0,4.0)-- (0.0,2.0);
\draw [line width=1pt](8.0,4.0)-- (8.0,2.0);
\draw [line width=1pt](16.0,4.0)-- (16.0,2.0);
\draw [line width=1pt](24.0,4.0)-- (24.0,2.0);
\draw [line width=1pt][shift={(3.0,2.0)}] plot[domain=0.0:3.141592653589793,variable=\t]({1.0*1.0*cos(\t r)+-0.0*1.0*sin(\t r)},{0.0*1.0*cos(\t r)+1.0*1.0*sin(\t r)});
\draw [line width=1pt][shift={(11.0,2.0)}] plot[domain=0.0:3.141592653589793,variable=\t]({1.0*1.0*cos(\t r)+-0.0*1.0*sin(\t r)},{0.0*1.0*cos(\t r)+1.0*1.0*sin(\t r)});
\draw [line width=1pt][shift={(19.0,2.0)}] plot[domain=0.0:3.141592653589793,variable=\t]({1.0*1.0*cos(\t r)+-0.0*1.0*sin(\t r)},{0.0*1.0*cos(\t r)+1.0*1.0*sin(\t r)});
\draw [line width=1pt][shift={(27.0,2.0)}] plot[domain=0.0:3.141592653589793,variable=\t]({1.0*1.0*cos(\t r)+-0.0*1.0*sin(\t r)},{0.0*1.0*cos(\t r)+1.0*1.0*sin(\t r)});
\draw (-0.6,5.2) node[anchor=north west] {$\ps$};
\draw (7.4,5.2) node[anchor=north west] {$\ps$};
\draw (15.4,5.2) node[anchor=north west] {$\ps$};
\draw (23.4,5.2) node[anchor=north west] {$\ps$};
\draw [line width=1pt](4.0,0.0)-- (4.0,-2.0);
\draw [line width=1pt](12.0,0.0)-- (12.0,-2.0);
\draw [line width=1pt](20.0,0.0)-- (20.0,-2.0);
\draw [line width=1pt](28.0,0.0)-- (28.0,-2.0);
\draw (3.4,-1.9) node[anchor=north west] {$\ps$};
\draw (11.4,-1.9) node[anchor=north west] {$\ps$};
\draw (19.4,-1.9) node[anchor=north west] {$\ps$};
\draw (27.4,-1.9) node[anchor=north west] {$\ps$};
\draw (2.1,-0.2) node[anchor=north west] {$\beta_1^{-1}$};
\draw (10.1,-0.2) node[anchor=north west] {$\beta_2^{-1}$};
\draw (18.1,-0.2) node[anchor=north west] {$\beta_3^{-1}$};
\draw (26.1,-0.2) node[anchor=north west] {$\beta_4^{-1}$};
\draw [fill=black] (2.0,2.0) circle (3pt);
\draw [fill=black] (0.0,2.0) circle (3pt);
\draw [fill=black] (10.0,2.0) circle (3pt);
\draw [fill=black] (8.0,2.0) circle (3pt);
\draw [fill=black] (12.0,2.0) circle (3pt);
\draw [fill=black] (12.0,0.0) circle (3pt);
\draw [fill=black] (4.0,2.0) circle (3pt);
\draw [fill=black] (4.0,0.0) circle (3pt);
\draw [fill=black] (18.0,2.0) circle (3pt);
\draw [fill=black] (16.0,2.0) circle (3pt);
\draw [fill=black] (20.0,2.0) circle (3pt);
\draw [fill=black] (20.0,0.0) circle (3pt);
\draw [fill=black] (26.0,2.0) circle (3pt);
\draw [fill=black] (24.0,2.0) circle (3pt);
\draw [fill=black] (28.0,2.0) circle (3pt);
\draw [fill=black] (28.0,0.0) circle (3pt);
\draw [fill=black] (0.0,4.0) circle (3pt);
\draw [fill=black] (8.0,4.0) circle (3pt);
\draw [fill=black] (16.0,4.0) circle (3pt);
\draw [fill=black] (24.0,4.0) circle (3pt);
\draw [fill=black] (4.0,-2.0) circle (3pt);
\draw [fill=black] (12.0,-2.0) circle (3pt);
\draw [fill=black] (20.0,-2.0) circle (3pt);
\draw [fill=black] (28.0,-2.0) circle (3pt);
\end{tikzpicture}
\right)^T \]
\end{center}
\noindent By stretching the diagrams the group elements cancel out, and we are left with the diagonal $\Delta^4=\langle \mj_Q,\mj_Q,\mj_Q,\mj_Q\rangle$.

\subsection{Entanglement Swapping}
The idea of this protocol is to, starting with two pairs of mutually entangled qubits in EPR-states, obtain again two pairs of entangled states, but with different pairing. Assume Alice, as well as Bob, share a single EPR-pair with a third person, named Charlie. Then Charlie performs a  measurement on his qubits, and via classical communication transfers information on his outcomes to other parties, upon which a unitary correction is applied. Net result of this protocol is that Alice and Bob share an entangled EPR-pair, while Charlie is left with another EPR-pair. We thus say that the entanglement is swapped. A complete description of this protocol in terms of categorical quantum mechanics is presented in \cite[Theorem~9.3]{AC04}. Again, as in Section~\ref{teleportation}, by relying on the equalities~\ref{41}-\ref{43}, one may completely neglect the role of scalars and just check the commutativity of the diagram~\ref{Diagram:ES} below for the correctness of this protocol.

Let
$\tau\colon Q_d^\ast\otimes (4\cdot ((Q_a\otimes Q_b^\ast)\otimes Q_c)) \to 4\cdot(Q_d^\ast\otimes ((Q_a\otimes Q_b*)\otimes Q_c))$ and $\upsilon\colon (4\cdot ((Q_a\otimes Q_b^\ast))\otimes Q_c \to 4\cdot ((Q_a\otimes Q_b^\ast)\otimes Q_c)$ be distributivity isomorphisms, and let
\begin{align*}
\gamma_i &= (\beta_i)_\ast, \\
\mathrm{P}_i &= \ulcorner \gamma_i \urcorner \circ \llcorner \beta_i \lrcorner, \\
\zeta_i^{ac} &= \bigoplus\nolimits_{i=1}^{i=4} ((\mj_b^\ast\otimes \beta_i)\otimes (\mj_d^\ast\otimes \beta_i^{-1})), \\
\Theta_{ab} &= \mj_d^\ast \otimes (\langle\mathrm{P}_i \rangle_{i=1}^{i=4}\otimes\mj_{c}), \\
\varphi&= (4\cdot ((\sigma_{ab}\otimes\mj_{dc})\circ \alpha^{-1}_{ab,d,c}\circ (\sigma_{d,ab}\otimes \mj_c)\circ \alpha_{d,ab,c})) \circ \tau\circ (\mj_d\otimes \upsilon), \\
\Omega_{ab} &= \langle\ulcorner \mj_{ba} \urcorner \circ \ulcorner \mj_{dc} \urcorner\rangle_{i=1}^{i=4}.
\end{align*}
The commutativity of the diagram from \cite[Theorem 9.3]{AC04} justifies the correctness of the entanglement swapping protocol. By factoring the scalars out from the legs, it reduces to the following diagram.

\begin{equation} \label{Diagram:ES}
\begin{tikzcd}[row sep=1.4cm, every label/.append style = {font = \small}]
I\otimes I \arrow[ddddd, "\Omega_{ab}"', bend right=80] \arrow[d, "\displaystyle\ulcorner\mj_{da}\urcorner \otimes \ulcorner\mj_{bc}\urcorner"', "\text{ produce EPR-pairs}"] \\
(Q_d^\ast\otimes Q_a)\otimes (Q_b^\ast\otimes Q_c) \arrow[d, "(\mj_{d}\otimes \alpha_{a,b,c})\circ \alpha^{-1}_{d,a,bc}"', "\text{ spatial delocation}"]\\
Q_d^\ast \otimes ((Q_a\otimes Q_b^\ast)\otimes Q_c) \arrow[d, "\Theta_{ab}"', "\text{ Bell-base measurement}"] \\
Q_d^\ast\otimes ((4\cdot (Q_a\otimes Q_b^\ast))\otimes Q_c) \arrow[d, "\varphi"', "\text{ classical communication}"] \\
4\cdot ((Q_b^\ast\otimes Q_a)\otimes (Q_d^\ast\otimes Q_c)) \arrow[d, "\zeta_i^{ac}"', "\text{ unitary correction}"] \\
4\cdot ((Q_b^\ast\otimes Q_a)\otimes (Q_d^\ast\otimes Q_c))
\end{tikzcd}
\end{equation}

The right-hand side of this diagram is represented in $1\cobG^{\oplus}$ by the $4\times 1$ matrix whose $i1$-entry is the following $\mathfrak{G}$-cobordism (note that we ignore associativity and distributivity isomorphisms since they are identities).
\begin{center}
\begin{tikzpicture}[scale=0.7, line cap=round,line join=round,x=1.0cm,y=1.0cm]
\draw [shift={(1.0,1.0)}][line width=1pt] plot[domain=0.0:3.141592653589793,variable=\t]({1.0*1.0*cos(\t r)+-0.0*1.0*sin(\t r)},{0.0*1.0*cos(\t r)+1.0*1.0*sin(\t r)});
\draw [shift={(1.0,3.5)}][line width=1pt] plot[domain=3.141592653589793:6.283185307179586,variable=\t]({1.0*1.0*cos(\t r)+-0.0*1.0*sin(\t r)},{0.0*1.0*cos(\t r)+1.0*1.0*sin(\t r)});
\draw [line width=1pt](0.6,3.3) node[anchor=north west] {$\beta_i$};
\draw [line width=1pt](0.5,2) node[anchor=north west] {$\beta_i^{-1}$};
\draw [line width=1pt](-2.0,3.5)-- (-2.0,1.0);
\draw [line width=1pt](4.0,3.5)-- (4.0,1.0);
\draw [line width=1pt][shift={(-1.0,3.5)}] plot[domain=0.0:3.141592653589793,variable=\t]({1.0*1.0*cos(\t r)+-0.0*1.0*sin(\t r)},{0.0*1.0*cos(\t r)+1.0*1.0*sin(\t r)});
\draw [line width=1pt][shift={(3.0,3.5)}] plot[domain=0.0:3.141592653589793,variable=\t]({1.0*1.0*cos(\t r)+-0.0*1.0*sin(\t r)},{0.0*1.0*cos(\t r)+1.0*1.0*sin(\t r)});
\draw [line width=1pt](-2.0,-2.0)-- (-2.0,-3.5);
\draw [line width=1pt](0.0,-2.0)-- (0.0,-3.5);
\draw [line width=1pt](2.0,-2.0)-- (2.0,-3.5);
\draw [line width=1pt](4.0,-2.0)-- (4.0,-3.5);
\draw [line width=1pt](0.0,1.0)-- (-2.0,-0.5);
\draw [line width=1pt](2.0,1.0)-- (0.0,-0.5);
\draw [line width=1pt](0.0,-0.5)-- (-2.0,-2.0);
\draw [line width=1pt](-2.0,1.0)-- (2.0,-0.5);
\draw [line width=1pt](4.0,1.0)-- (4.0,-0.5);
\draw [line width=1pt](4.0,-0.5)-- (4.0,-2.0);
\draw [line width=1pt](-2.0,-0.5)-- (0.0,-2.0);
\draw [line width=1pt](2.0,-0.5)-- (2.0,-2.0);
\draw (3.9,4.4) node[anchor=north west] {$\ulcorner 1_{da}\urcorner \otimes \ulcorner 1_{bc}\urcorner$};
\draw (5,2.7) node[anchor=north west] {$\Theta_{ab}$};
\draw (5,-2.5) node[anchor=north west] {$\zeta_i^{ac}$};
\draw (4.4,0.7) node[anchor=north west] {$\sigma_{d,ab}\otimes 1_c$};
\draw (4.4,-0.8) node[anchor=north west] {$\sigma_{a,b}\otimes 1_{dc}$};
\draw [dash pattern=on 3pt off 3pt] (-3.0,3.5)-- (7.0,3.5);
\draw [dash pattern=on 3pt off 3pt] (7.0,1.0)-- (-3.0,1.0);
\draw [dash pattern=on 3pt off 3pt] (-3.0,-0.5)-- (7.0,-0.5);
\draw [dash pattern=on 3pt off 3pt] (-3.0,-2.0)-- (7.0,-2.0);
\draw (-0.7,-2.5) node[anchor=north west] {$\beta_i$};
\draw (2.9,-2.5) node[anchor=north west] {$\beta_i^{-1}$};
\draw (-2.6,3.6) node[anchor=north west] {$\ms$};
\draw (1.4,3.6) node[anchor=north west] {$\ms$};
\draw (-2.3,1) node[anchor=north west] {$\ms$};
\draw (1.7,1) node[anchor=north west] {$\ms$};
\draw (-0.3,-0.4) node[anchor=north west] {$\ms$};
\draw (1.4,-0.4) node[anchor=north west] {$\ms$};
\draw (-2.6,-1.9) node[anchor=north west] {$\ms$};
\draw (1.4,-1.9) node[anchor=north west] {$\ms$};
\draw (-2.3,-3.5) node[anchor=north west] {$\ms$};
\draw (1.7,-3.5) node[anchor=north west] {$\ms$};
\draw (-0.5,3.6) node[anchor=north west] {$\ps$};
\draw (3.4,3.6) node[anchor=north west] {$\ps$};
\draw (-0.3,1) node[anchor=north west] {$\ps$};
\draw (3.4,1) node[anchor=north west] {$\ps$};
\draw (-2.3,-0.5) node[anchor=north west] {$\ps$};
\draw (3.4,-0.5) node[anchor=north west] {$\ps$};
\draw (-0.5,-1.9) node[anchor=north west] {$\ps$};
\draw (3.4,-1.9) node[anchor=north west] {$\ps$};
\draw (-0.3,-3.5) node[anchor=north west] {$\ps$};
\draw (3.7,-3.5) node[anchor=north west] {$\ps$};
\begin{scriptsize}
\draw [fill=black] (0.0,1.0) circle (2pt);
\draw [fill=black] (2.0,1.0) circle (2pt);
\draw [fill=black] (2.0,3.5) circle (2pt);
\draw [fill=black] (0.0,3.5) circle (2pt);
\draw [fill=black] (-2.0,3.5) circle (2pt);
\draw [fill=black] (-2.0,1.0) circle (2pt);
\draw [fill=black] (4.0,3.5) circle (2pt);
\draw [fill=black] (4.0,1.0) circle (2pt);
\draw [fill=black] (2.0,-0.5) circle (2pt);
\draw [fill=black] (-2.0,-0.5) circle (2pt);
\draw [fill=black] (0.0,-0.5) circle (2pt);
\draw [fill=black] (4.0,-0.5) circle (2pt);
\draw [fill=black] (4.0,-2.0) circle (2pt);
\draw [fill=black] (2.0,-2.0) circle (2pt);
\draw [fill=black] (0.0,-2.0) circle (2pt);
\draw [fill=black] (-2.0,-2.0) circle (2pt);
\draw [fill=black] (-2.0,-3.5) circle (2pt);
\draw [fill=black] (0.0,-3.5) circle (2pt);
\draw [fill=black] (2.0,-3.5) circle (2pt);
\draw [fill=black] (4.0,-3.5) circle (2pt);
\end{scriptsize}
\end{tikzpicture}
\end{center}
By stretching the above diagram and cancelling $\beta_i$ and $\beta_i^{-1}$, we are left with the following $\mathfrak{G}$-cobordism.

\begin{center}
\begin{tikzpicture}[scale=0.7,line cap=round,line join=round,x=1.0cm,y=1.0cm]
\draw [line width=1pt][shift={(9.0,1.0)}] plot[domain=0.0:3.141592653589793,variable=\t]({1.0*1.0*cos(\t r)+-0.0*1.0*sin(\t r)},{0.0*1.0*cos(\t r)+1.0*1.0*sin(\t r)});
\draw [line width=1pt][shift={(12.0,1.0)}] plot[domain=0.0:3.141592653589793,variable=\t]({1.0*1.0*cos(\t r)+-0.0*1.0*sin(\t r)},{0.0*1.0*cos(\t r)+1.0*1.0*sin(\t r)});
\draw (7.7,1) node[anchor=north west] {$\ms$};
\draw (10.7,1) node[anchor=north west] {$\ms$};
\draw (9.7,1) node[anchor=north west] {$\ps$};
\draw (12.7,1) node[anchor=north west] {$\ps$};
\begin{scriptsize}
\draw [fill=black] (8.0,1.0) circle (2pt);
\draw [fill=black] (10.0,1.0) circle (2pt);
\draw [fill=black] (11.0,1.0) circle (2pt);
\draw [fill=black] (13.0,1.0) circle (2pt);
\end{scriptsize}
\end{tikzpicture}
\end{center}

On the other side, $\Omega_{ab}$ is represented in $1\cobG^{\oplus}$ by $4\times 1$ matrix, whose $i1$-entry is exactly the above $\mathfrak{G}$-cobordism. Due to Proposition \ref{coherence}, this proves the commutativity of the diagram \ref{Diagram:ES}.

\subsection{Superdense coding}\label{superdense}
In this section, we will apply our diagrammatic verification to another protocol, called a superdense coding, \cite{NC10} (sometimes referred to as a dense coding). This quantum algorithm can be considered as an opposite of the quantum teleportation. The idea is to transfer some amount of classical information, using qubits. A review of this protocol can be found in \cite{V12}, where its validity was shown in a similar manner.

The validity of this protocol is expressed in the categorical setting by the commutativity of a diagram in which some special scalars, namely traces of some arrows, occur. Every compact closed category can be lifted to the traced category by a suitable definition of a categorical trace. This can be achieved as follows. Let $f\colon a\to a$ be an arrow in a compact closed category. The scalar $\mathrm{Tr}(f)\colon I\rightarrow I$ is defined as\footnote{More generaly, categorical trace corresponds to the partial trace in Hilbert space picture, though we will not review this here, as our interest lies only in pure states.}
\begin{equation}
    \mathrm{Tr}(f)= \varepsilon_a \circ (f\otimes a^\ast) \circ \sigma_{a^*,a}\circ \eta_a.
\end{equation}
In terms of diagrams, we have
\begin{center}
    \begin{tikzpicture}[scale=0.7][line cap=round,line join=round,>=triangle 45,x=1.0cm,y=1.0cm]
\draw [shift={(2.,2.)},line width=1.pt]  plot[domain=0.:3.141592653589793,variable=\t]({1.*1.*cos(\t r)+0.*1.*sin(\t r)},{0.*1.*cos(\t r)+1.*1.*sin(\t r)});
\draw [line width=1.pt] (1.,2.)-- (3.,1.);
\draw [line width=1.pt] (3.,2.)-- (1.,1.);
\draw [line width=1.pt] (1.,1.)-- (1.,0.);
\draw [shift={(2.,0.)},line width=1.pt]  plot[domain=3.141592653589793:6.283185307179586,variable=\t]({1.*1.*cos(\t r)+0.*1.*sin(\t r)},{0.*1.*cos(\t r)+1.*1.*sin(\t r)});
\draw [line width=1.pt] (3.,0.)-- (3.,1.);
\draw [line width=1.pt] (9.,1.) circle (1.40356688476182cm);
\draw (0.3988007525566426,0.6481020297918041) node[anchor=north west] {$f$};
\draw (7.055729924061316,0.9337258383990148) node[anchor=north west] {$f$};
\draw [->,>=stealth][line width=1.pt] (4.175307573377374,1.0206548236272963)-- (6.534808601002162,1.0206548236272963);
\end{tikzpicture}
\end{center}

A category appropriate for the superdense coding requires the same structure as in the first two protocols. Moreover, the following conditions must be satisfied. If $i\neq j$, then $\mathrm{Tr}(\beta_i\beta_j^\dagger)=0_{I,I}$, and $\mathrm{Tr}(\mj_Q)\neq 0_{I,I}$ (see Appendix \ref{skalari} for the details why we demand this condition to be satisfied). With this in mind the arrow $\Xi\colon I \to 16\cdot I$ defined as
\[
\langle \mathrm{Tr}(\beta_1\beta_1^\dagger), \mathrm{Tr}(\beta_1\beta_2^\dagger), \mathrm{Tr}(\beta_1\beta_3^\dagger), \mathrm{Tr}(\beta_1\beta_4^\dagger),\ldots,
\mathrm{Tr}(\beta_4\beta_1^\dagger),
\mathrm{Tr}(\beta_4\beta_2^\dagger), \mathrm{Tr}(\beta_4\beta_3^\dagger), \mathrm{Tr}(\beta_4\beta_4^\dagger)\rangle
\]
is actually $\langle t,0,0,0,0,t,0,0,0,0,t,0,0,0,0,t\rangle$, for $t=\mathrm{Tr}(\mj_Q)$. The assumption above also enables Bob to make a distinction between the four quadruples of scalars in this row.

Our task is to show that the following diagram, which verifies the superdense coding protocol, commutes.
\begin{equation}
\begin{tikzcd}[row sep=1.4cm, every label/.append style = {font = \small}]
I \arrow[dddd, "\Xi"', bend right=90] \arrow[d, "\displaystyle\ulcorner\mj_{ab}\urcorner"', "\text{ preparation of EPR-pair}"] \\
Q_a^\ast\otimes Q_b \arrow[d, "\displaystyle \langle (\beta_i)_\ast \rangle_{i=1}^{i=4}\otimes\mj_b"', "\text{ selection of classical information}"]\\
(4\cdot Q_a^\ast)\otimes Q_b \arrow[d, "(4\cdot\sigma_{ab})\circ\upsilon_{4a,b}"', "\text{ spatial delocation}"] \\
4\cdot(Q_b \otimes Q_a^\ast)\arrow[d, "\displaystyle \langle\llcorner \beta_i^{ab} \lrcorner\rangle_{i=1}^{i=4}"', "\text{ observation}"] \\
16\cdot I
\end{tikzcd}
\end{equation}
Here, again, the first step is the EPR-pair production, achieved by a cap diagram.
\begin{center}
\begin{tikzpicture}[scale=0.7,line cap=round,line join=round,x=1.0cm,y=1.0cm]
\draw [shift={(2.,1.)},line width=1.pt]  plot[domain=0.:3.141592653589793,variable=\t]({1.*1.*cos(\t r)+0.*1.*sin(\t r)},{0.*1.*cos(\t r)+1.*1.*sin(\t r)});
\draw (0.7,1) node[anchor=north west] {$\ms$};
\draw (2.7,1) node[anchor=north west] {$\ps$};
\draw [fill=black] (1,1) circle (2pt);
\draw [fill=black] (3,1) circle (2pt);
\end{tikzpicture}
\end{center}
\vspace{-3mm}
One qubit is located at Alice's point, and another at Bob's. Alice then applies an unitary transformation to her qubit, depending on the classical infromation she wants to communicate. This is achieved by $\langle (\beta_i)_\ast \rangle_{i=1}^{i=4}\otimes\mj_b$
By composing the first two arrows, we get a $4\times 1$ matrix, whose $i1$-entry is given by a following arrow.
\begin{center}
\begin{tikzpicture}[scale=0.7][line cap=round,line join=round,x=1.0cm,y=1.0cm]
\draw [shift={(2.,2.)},line width=1.pt]  plot[domain=0.:3.141592653589793,variable=\t]({1.*1.*cos(\t r)+0.*1.*sin(\t r)},{0.*1.*cos(\t r)+1.*1.*sin(\t r)});
\draw [line width=1.pt] (1.,2.)-- (1.,1.);
\draw [line width=1.pt] (3.,2.)-- (3.,1.);
\draw (-0.1,2) node[anchor=north west] {$\beta_i^{-1}$};
\draw (0.7,1) node[anchor=north west] {$\ms$};
\draw (2.7,1) node[anchor=north west] {$\ps$};
\draw (1,2.3) node[anchor=north west] {$\ms$};
\draw (3,2.3) node[anchor=north west] {$\ps$};
\begin{scriptsize}
\draw [fill=black] (1,1) circle (2pt);
\draw [fill=black] (3,1) circle (2pt);
\draw [fill=black] (1,2) circle (2pt);
\draw [fill=black] (3,2) circle (2pt);
\end{scriptsize}
\end{tikzpicture}
\end{center}
\vspace{-3mm}
Spatial delocation is represented by a transposition, and fter its application we obtain a matrix with $i1$-entry given by
\begin{center}
\begin{tikzpicture}[scale=0.7][line cap=round,line join=round,x=1.0cm,y=1.0cm]
\draw [shift={(0.,2.)},line width=1pt]  plot[domain=0.:3.141592653589793,variable=\t]({1.*1.*cos(\t r)+0.*1.*sin(\t r)},{0.*1.*cos(\t r)+1.*1.*sin(\t r)});
\draw [line width=1.pt] (-1.,2.)-- (-1.,1.);
\draw [line width=1.pt] (1.,2.)-- (1.,1.);
\draw [line width=1.pt] (-1.,1.)-- (1.,0.);
\draw [line width=1.pt] (1.,1.)-- (-1.,0.);
\draw (0.7,0) node[anchor=north west] {$\ms$};
\draw (-1.3,0) node[anchor=north west] {$\ps$};
\draw (0.7,1) node[anchor=north west] {$\ps$};
\draw (-1.3,1) node[anchor=north west] {$\ms$};
\draw (1,2.3) node[anchor=north west] {$\ps$};
\draw (-1,2.3) node[anchor=north west] {$\ms$};
\draw (-2.1,2) node[anchor=north west] {$\beta_i^{-1}$};
\begin{scriptsize}
\draw [fill=black] (1,0) circle (2pt);
\draw [fill=black] (-1,0) circle (2pt);
\draw [fill=black] (1,1) circle (2pt);
\draw [fill=black] (-1,1) circle (2pt);
\draw [fill=black] (1,2) circle (2pt);
\draw [fill=black] (-1,2) circle (2pt);
\end{scriptsize}
\end{tikzpicture}
\end{center}
\vspace{-3mm}
Finally, Alice sends her qubit to Bob, who performs an entangled state measurement, given by a suitable coname. The result is  a $16\times 1$ matrix, whose $(4(i-1)+j)1$- element is given by the $\mathfrak{G}$-circle
\begin{center}
\begin{tikzpicture}[scale=0.7][line cap=round,line join=round,x=1.0cm,y=1.0cm]
\draw [shift={(4.,2.)},line width=1.pt]  plot[domain=0.:3.141592653589793,variable=\t]({1.*1.*cos(\t r)+0.*1.*sin(\t r)},{0.*1.*cos(\t r)+1.*1.*sin(\t r)});
\draw [line width=1.pt] (3.,2.)-- (3.,1.);
\draw [line width=1.pt] (5.,2.)-- (5.,1.);
\draw [shift={(4.,1.)},line width=1.pt]  plot[domain=3.141592653589793:6.283185307179586,variable=\t]({1.*1.*cos(\t r)+0.*1.*sin(\t r)},{0.*1.*cos(\t r)+1.*1.*sin(\t r)});
\draw (1.9,1.8) node[anchor=north west] {$\beta_i^{-1}$};
\draw (3.5,0) node[anchor=north west] {$\beta_j$};
\end{tikzpicture}
\end{center}
\vspace{-3mm}
and the same matrix is obtained by interpreting the arrow $\Xi\colon I \to 16\cdot I$ in the category $1\cobG^{\oplus}$. The additional assumptions on the compact closed structure, listed in the paragraph where the arrow $\Xi$ is defined, enable Bob to distinguish between different Alice's messages.

\section{Omitting labels}\label{labels}

This section, in an informal way, illustrates possibility of elimination of labels assigned to cobordisms in graphical categories, introduced in Section~\ref{graficki}, that serve as models for syntactical categories, introduced in Section~\ref{free category}. In the case when we have some additional equations concerning the unitary endomorphisms from $\Gamma$, it may produce a finite group of automorphisms of $p$. Then it is not necessary to increase the dimension of cobordisms since such a group appears as a subgroup of a symmetric group $S_n$, for some $n$. One can represent $p$ by $n$ points and interpret every element of $\Gamma$ as the corresponding permutation.

For example, the teleportation protocol requires the dihedral group $D_4$, which is a subgroup of $S_4$. This means that it is sufficient to define the syntactical category $\mathcal{F}^\dagger$ so that $\Gamma$ is a set of two elements that satisfy, besides the equalities \ref{33} and \ref{34}, the equalities of the standard presentation of $D_4$. In this case, the category $1\cobG^\oplus$ should be replaced by a thickened version. This means that before labeling, every segment is replaced by four parallel threads (the diagram is thickened), and every label is replaced by the corresponding permutation of four threads. The two elements of $\Gamma$ correspond to the following permutations.

\begin{center}
\includegraphics{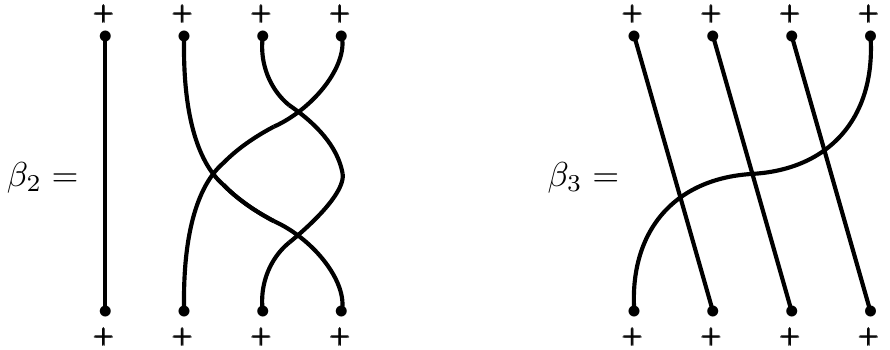}
\end{center}

However, there is no possibility to interpret an infinite group in such a way.

By increasing dimension by one, according to remark given in the penultimate paragraph on page~60 (after Proposition~1.4.9) of \cite{K03}, the situation remains the same. This remark says that the only invertible 2-cobordisms are the permutation cobordisms.

Hence, for the interpretation of an infinite group generated by $\Gamma$, one has to consider 3-cobordisms. We rely here on \cite[Definition~2.3]{J18} in order to introduce cobordisms that replace $\mathfrak{G}$-cobordisms from Section~\ref{graficki}. Namely, for every orientation preserving homeomorphism $h\colon \Sigma_g\to\Sigma_g$, where $\Sigma_g$ is a closed oriented surface of genus $g$, there is a cobordism $(\Sigma_g\times I, f_0,f_1)$, where $f_0(x)=(x,0)$ and $f_1(x)=(h(x),1)$. Two such cobordisms, corresponding to homeomorphisms $h$ and $h'$ respectively, are equivalent if and only if $h$ and $h'$ are pseudo-isotopic. According to \cite{E66}, this is equivalent to the fact that $h$ and $h'$ are isotopic.

By applying technique introduced in \cite{L62}, the cobordism $(\Sigma_g\times I, f_0,f_1)$ is equivalent to the cobordism $(M, f_0,f_1')$, where $M$ is $\Sigma_g\times I$ with some extra surgery, and $f_1'(x)=(x,1)$. Here we will illustrate just the case of the group freely generated by one generator.
In the case of a group freely generated by more generators, according to comments from the preceding paragraph, the results concerning free subgroups of the mapping class groups of surfaces, obtained in \cite{I92}, \cite{I96} and \cite{AAS07} are relevant.

In our example we suggest to replace the $\mathfrak{G}$-segment labeled by the generator of $\mathfrak{G}$, i.e. a $\mathfrak{G}$-cobordism introduced in Section~\ref{graficki}, by a 3-cobordism $C$ obtained as follows. For $T^2$ being the 2-dimensional torus, the underlying manifold of $C$ is $T^2\times I$ with some additional surgery. Moreover, for $i\in\{0,1\}$, the embeddings $f_i$ are of the form $f_i(x)=(x,i)$. In order to present such a cobordism we use the diagrammatical language introduced in \cite{NPZ} (see Figure~\ref{generator} for a presentation of $C$).

\begin{figure}[!h]
\centering
\includegraphics{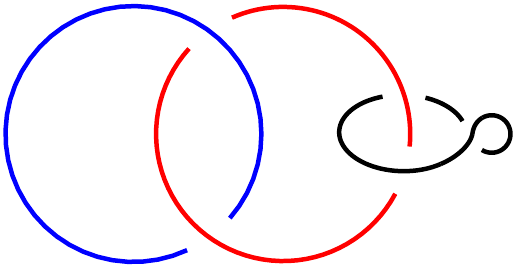}
\caption{Cobordism $C$}
\label{generator}
\end{figure}

Roughly speaking, tubular neighbourhoods of the red and the blue circle are removed from $S^3$ and the surgery along the black unknot is performed. Note that the twist of this unknot indicates the framing 1 of this surgery component. We refer to \cite{NPZ} for details of the interpretation of such diagrams. The rules for composing diagrams say that $C\circ C$ is presented by the diagram at the left-hand side of Figure~\ref{inverz} and that the $\mathfrak{G}$-segment labeled by the inverse of the generator of $\mathfrak{G}$ should be replaced by the cobordism $C^{-1}$ illustrated at the right-hand side of Figure~\ref{inverz}.

\begin{figure}[!h]
\centering
\includegraphics{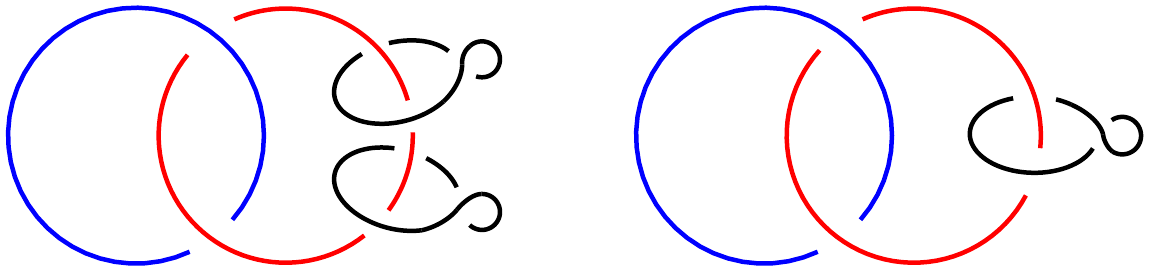}
\caption{Cobordisms $C\circ C$ and $C^{-1}$}
\label{inverz}
\end{figure}

The same rules say that the $\mathfrak{G}$-circle labeled by the product of the generator with itself should be replaced by the cobordism (closed 3-manifold) illustrated in Figure~\ref{krug}.

\begin{figure}[!ht]
\includegraphics{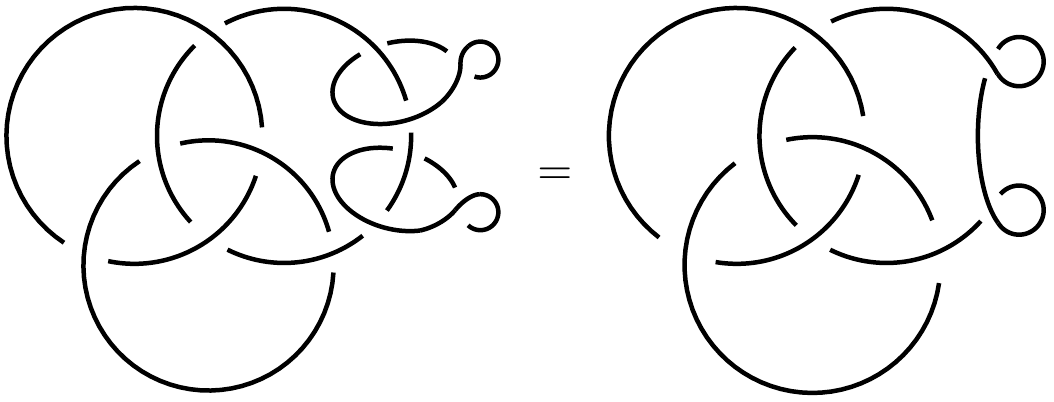}
\caption{Cobordism corresponding to the labeled $\mathfrak{G}$-circle}
\label{krug}
\end{figure}

Though this switching to dimension three could be less practical at some points, it could bring some new insight to the subject through the variant of Kirby calculus introduced in \cite{FGOP}. Our plan for a future work is to investigate this 3-dimensional calculus.

\section{Concluding remarks}\label{conclusion}

After the introduction of categorical quantum mechanics, it is natural to seek for a different dagger compact closed categories with biproducts, in order to check whether they can sustain quantum protocols, as quantum teleportation. The possible complication is the existence of a base. Abstractly, base can be defined using biproducts: we demand existence of unitary arrows $I\oplus I\rightarrow Q$. The problem with the category of cobordisms is that there does not seem to be enough options to construct the desired unitary morphism. This was alluded, in a slightly different context, in \cite{BA04}. Luckily, in order to verify protocols as quantum teleportation, it is not mandatory to use the described morphism.

Furthermore, in low dimensions, it is hopeless to try to accommodate different unitary transformations present in quantum protocols as different cobordisms. In order to heal this problem, we introduced a group structure $\mathfrak{G}$. We believe that the approach suggested in the preceding section could provide a solution for these problems.

Of course, this raises some conceptual questions. First, by identifying the qubit state space with $\ps$ (or $\ms$), we are not able to use our graphical language to define states, i.e.\ morphisms of the form $I\rightarrow Q$. In all mentioned quantum protocols, this was not an issue, as we used names to create entangled states, and this can be seen in $1\mathbf{Cob}_{\mathfrak{G}}$ language. In order to circumvent this issue, one could increase the dimension of cobordisms as suggested in the preceding section, or to take zero-dimensional spheres, i.e. the two element sequences $\ps\hspace{1mm}\ms$ to represent state spaces. Then, one has the possibility to introduce morphisms that define states. Also, we can use this new type of qubits to define measurements on a single qubit, not just on an entangled pair. Considerations of this type could be of interest when dealing with single-particle protocols \cite{DEL18}.

We conclude this section with a comment concerning the generality of quantum protocols brought by replacing the Hilbert spaces by objects of a compact closed category. It is known that (with minor provisos) all the 1-dimensional topological quantum field theories, i.e.\ functors from the category $1\cob$ to the category of finite dimensional vector spaces over a field, are faithful according to \cite{T19}. However, this does not mean that the whole $1\cobG^\oplus$ could be faithfully represented by matrices over a field. On the other hand, since protocols do not use the full strength of $1\cobG^\oplus$, one could expect that some could be verified  by relying on the matrix calculus (working again in the skeleton of $\fdhilb$ with chosen bases). This could be an advantage concerning computational issues of the problem.

\appendix
\renewcommand{\thefigure}{A\arabic{figure}}
\renewcommand{\theequation}{\Alph{section}.\arabic{equation}}
\section*{Appendix}

\section{The language and the equations for dagger compact closed categories with dagger biproducts}\label{language}

Our choice of a language for dagger compact closed categories with dagger biproducts is the one in which enrichment over $\cmd$ is primitive and not derived from the biproduct structure. Such a language is siutable for the proofs of our results. A dagger compact closed category with dagger biproducts $\mathcal{A}$ consists of a set of objects and a set of arrows. There are two functions (\emph{source} and \emph{target}) from the set of arrows to the set of objects of $\mathcal{A}$. For every object $a$ of $\mathcal{A}$ there is the identity arrow $\mj_a\colon a\to a$. The set of objects includes two distinguished objects $I$ and $0$. Arrows $f\colon a\to b$ and $g\colon b\to c$ \emph{compose} to give $g\circ f\colon a\to c$, and arrows $f_1,f_2\colon a\to b$ \emph{add} to give $f_1+f_2\colon a\to b$. For every object $a$ of $\mathcal{A}$, there is the object $a^\ast$, and for every pair of objects $a$ and $b$ of $\mathcal{A}$, there are the objects $a\otimes b$ and $a\oplus b$. Also, for every arrow $f\colon a\to b$, there is the arrow $f^\dagger\colon b\to a$, and for every pair of arrows $f\colon a\to a'$ and $g\colon b\to b'$ there are the arrows $f\otimes g\colon a\otimes b\to a'\otimes b'$ and $f\oplus g\colon a\oplus b\to a'\oplus b'$. In $\mathcal{A}$ we have the following families of arrows indexed by its objects.

\begin{align*}
&\alpha_{a,b,c} \colon a\otimes(b\otimes c)\to (a\otimes b)\otimes c, \quad &&\alpha^{-1}_{a,b,c}\colon (a\otimes b)\otimes c\to a\otimes(b\otimes c),
\\[1ex]
& \lambda_a \colon I\otimes a\to a, \quad &&\lambda^{-1}_a\colon a\to I\otimes a,
\\[1ex]
& \sigma_{a,b}\colon a\otimes b\to b\otimes a,
\\[1ex]
& \eta_a\colon I\to a^\ast\otimes a, \quad &&\varepsilon_a\colon a\otimes a^\ast\to I
\\[1ex]
& \pi^1_{a,b} \colon a\oplus b\to a,\quad && \iota^1_{a,b}\colon a\to a\oplus b,
\\[1ex]
& \pi^2_{a,b}\colon a\oplus b\to b,\quad && \iota^2_{a,b}\colon b\to a\oplus b,
\\[1ex]
& 0_{a,b} \colon a\to b.
\end{align*}

The arrows of $\mathcal{A}$ should satisfy the following equalities:
\begin{equation}\label{1}
   f\circ \mj_a=f=\mj_{a'}\circ f,\quad (h\circ g)\circ f=h\circ(g\circ f),
\end{equation}
\begin{equation}\label{2}
   \mj_a\otimes \mj_b=\mj_{a\otimes b},\quad (f_2\otimes g_2)\circ(f_1\otimes g_1)= (f_2\circ f_1)\otimes(g_2\circ g_1),
\end{equation}

\begin{equation}\label{5}
   \begin{array}{c}((f\otimes g)\otimes h)\circ\alpha_{a,b,c}= \alpha_{a',b',c'}\circ(f\otimes(g\otimes h)),\\[1ex] \alpha^{-1}_{a,b,c}\circ\alpha_{a,b,c}=\mj_{a\otimes(b\otimes c)},\quad \alpha_{a,b,c}\circ\alpha^{-1}_{a,b,c}=\mj_{(a\otimes b)\otimes c},
   \end{array}
\end{equation}
\begin{equation}\label{6}
   f\circ\lambda_a=\lambda_{a'}\circ(I\otimes f),\quad \lambda^{-1}_a\circ\lambda_a=\mj_{I\otimes a},\quad \lambda_a\circ\lambda^{-1}_a=\mj_a,
\end{equation}
\begin{equation}\label{7}
   (g\otimes f)\circ\sigma_{a,b}=\sigma_{a',b'}\circ(f\otimes g),\quad \sigma_{b,a}\circ\sigma_{a,b}=\mj_{a\otimes b},
\end{equation}

\begin{equation}\label{19}
   \alpha_{a\otimes b,c,d}\circ\alpha_{a,b,c\otimes d}=(\alpha_{a,b,c}\otimes d)\circ\alpha_{a,b\otimes c,d}\circ (a\otimes\alpha_{b,c,d}),
\end{equation}
\begin{equation}\label{20}
   \lambda_{a\otimes b}=(\lambda_a\otimes b)\circ\alpha_{I,a,b},
\end{equation}
\begin{equation}\label{21}
   \alpha_{c,a,b}\circ\sigma_{a\otimes b,c}\circ \alpha_{a,b,c}= (\sigma_{a,c}\otimes b)\circ \alpha_{a,c,b} \circ (a\otimes \sigma_{b,c}),
\end{equation}

\begin{equation}\label{23}
(a^\ast\otimes \varepsilon)\circ\alpha^{-1}_{a^\ast,a,a^\ast}\circ (\eta\otimes a^\ast)=\sigma_{I,a^\ast},\quad (\varepsilon\otimes a)\circ\alpha_{a,a^\ast,a}\circ(a\otimes\eta)=\sigma_{a,I},
\end{equation}

\begin{equation}\label{16}
   f_1+(f_2+f_3)=(f_1+f_2)+f_3,\quad f_1+f_2=f_2+f_1,\quad f+0_{a,a'}=f,
\end{equation}
\begin{equation}\label{17}
   (g_1+g_2)\circ f=g_1\circ f + g_2\circ f,\quad g\circ(f_1+f_2)=g\circ f_1 + g\circ f_2,
\end{equation}
\begin{equation}\label{18}
   0_{a',b}\circ f=0_{a,b},\quad f\circ 0_{b,a}=0_{b,a'},
\end{equation}

\begin{equation}\label{3}
   \mj_a\oplus \mj_b=\mj_{a\oplus b},\quad (f_2\oplus g_2)\circ(f_1\oplus g_1)= (f_2\circ f_1)\oplus(g_2\circ g_1),
\end{equation}
\begin{equation}\label{10}
   (f\oplus g)\circ\iota^1_{a,b}=\iota^1_{a',b'}\circ f,\quad (f\oplus g)\circ\iota^2_{a,b}=\iota^2_{a',b'}\circ g,
\end{equation}
\begin{equation}\label{11}
   f\circ\pi^1_{a,b}=\pi^1_{a',b'}\circ (f\oplus g),\quad g\circ\pi^2_{a,b}=\pi^2_{a',b'}\circ (f\oplus g),
\end{equation}
\begin{equation}\label{13}
   \pi^1_{a,b}\circ\iota^1_{a,b}=\mj_a,\quad \pi^2_{a,b}\circ\iota^2_{a,b}=\mj_b,
\end{equation}
\begin{equation}\label{14}
   \pi^2_{a,b}\circ\iota^1_{a,b}=0_{a,b},\quad \pi^1_{a,b}\circ\iota^2_{a,b}=0_{b,a},
\end{equation}
\begin{equation}\label{15}
   \iota^1_{a,b}\circ \pi^1_{a,b} + \iota^2_{a,b}\circ \pi^2_{a,b}=\mj_{a\oplus b}.
\end{equation}

\begin{equation}\label{22}
   0_{0,0}=\mj_0,
\end{equation}

\begin{equation}\label{4}
   \mj_a^\dagger=\mj_a,\quad (g\circ f)^\dagger= f^\dagger\circ g^\dagger,\quad f^{\dagger\dagger}=f,
\end{equation}
\begin{equation}\label{28}
   (f\otimes g)^\dagger=f^\dagger\otimes g^\dagger,
\end{equation}
\begin{equation}\label{29}
   \alpha_{a,b,c}^\dagger=\alpha^{-1}_{a,b,c},\quad \lambda_a^\dagger=\lambda^{-1}_a,\quad \sigma_{a,b}^\dagger=\sigma_{b,a},
\end{equation}
\begin{equation}\label{30}
   \varepsilon^\dagger=\sigma_{a^\ast,a}\circ\eta,
\end{equation}
\begin{equation}\label{31}
   (\pi^1_{a,b})^\dagger=\iota^1_{a,b},\quad  (\pi^2_{a,b})^\dagger=\iota^2_{a,b}.
\end{equation}

The following equalities are derivable from \ref{1}-\ref{31}:
\begin{equation}\label{35}
   (f\oplus g)^\dagger = f^\dagger \oplus g^\dagger,
\end{equation}
\begin{equation}\label{36}
   (f+ g)^\dagger = f^\dagger + g^\dagger,\quad 0_{a,b}^\dagger=0_{b,a}
\end{equation}
\begin{equation}\label{25}
   f\otimes(g_1+g_2)=(f\otimes g_1)+(f\otimes g_2),\quad (f_1+f_2)\otimes g=(f_1\otimes g)+(f_2\otimes g),
\end{equation}
\begin{equation}\label{27}
   f\otimes 0_{b,b'}=0_{a\otimes b,a'\otimes b'}=0_{a,a'}\otimes g.
\end{equation}

\section{Scalars and Probability Amplitudes}\label{skalari}

As firmly laid, quantum mechanics is based on complex vector spaces (Hilbert spaces to be more precise). Implied in this structure is the notion of scalars, that correspond here to the field of complex numbers. In categorical language, one can define scalars more abstractly \cite{KL80,AC04}. A \emph{scalar} is a morphism $s:\hspace{1mm}I\rightarrow I$. It can be proved that the hom-set ${\rm Hom}\,(I,I)$, for a compact closed category, is a commutative monoid, therefore justifying further this structure's name.

In $1\cob$, the scalars correspond to closed, one-dimensional manifolds, and the only candidate for such a structure is a finite collection of circles $S^1$ (as denoted on the left-hand side of the following picture). In $1\cobG$, we have $\mathfrak{G}$-circles; topological circles dressed with group elements (right-hand side of the following picture). Due to the compact closed structure of this category, there is a natural interpretation of those circles. Namely, any compact closed category can be lifted to a traced category by a suitable definition of a categorical trace (see Section~\ref{superdense} for the definition).
\vspace{-4mm}
\begin{center}
\begin{tikzpicture}[scale=1.2][line cap=round,line join=round,>=triangle 45,x=1.0cm,y=1.0cm]
\draw [line width=1.pt] (-3.2607457074339212,7.203906133567978) circle (0.7574096972511294cm);
\draw [line width=1.pt] (1.2396821824888864,7.183309209586637) circle (0.7574096972511276cm);
\draw (0.2,6.8) node[anchor=north west] {$g_1$};
\draw (1.8,6.8) node[anchor=north west] {$g_3$};
\draw (1,8.3) node[anchor=north west] {$g_2$};
\end{tikzpicture}
\end{center}

That closed loops should be connected with traces in not limited to a categorical approach to quantum mechanics. Even when considering Feynman diagrams in quantum electrodynamics, fermions loops are accompanied by a trace in spinorial indices. Moreover, in TQFT, we are customed to the fact that closing manifold by gluing the outward future to inward past (if possible), results in a trace, that for a cylinder, i.e.\ the identity, simply gives the dimension of the respective Hilbert space.

Furthermore, as explained in \cite{CO05}, these traces correspond to the probability weights of different branches. This is further confirmed by a Hilbert space picture computations. Recall that one reason we have scalars (different from the multiplicative unit) is normalisation on states. In order to get the probabilistic interpretation, according to the Born rule, we must insist on normalised states. For a state $\beta_{00}\ket{0}\otimes\ket{0}+\beta_{01}\ket{0}\otimes\ket{1}+\beta_{10}\ket{1}\otimes\ket{0}+\beta_{11}\ket{1}\otimes\ket{1}$, we have its norm squared $|\beta_{00}|^2+|\beta_{01}|^2+|\beta_{10}|^2+|\beta_{11}|^2=\mathrm{Tr}(\beta^\dagger \beta)$, where $\beta$ is a $2\times 2$ matrix whose components are $\beta_{ij}$ constants. Therefore, we consider significance of traces in the usual sense.

When dealing with quantum protocols, one usually takes $\beta$ to be proportional to Pauli sigma matrices. (Extended) Pauli matrices are defined as
\begin{equation}
   \sigma_0=
   \begin{pmatrix}
  1&0\\
  0&1
   \end{pmatrix},
   \hspace{3mm}
   \sigma_1= \begin{pmatrix}
    0&1\\
    1&0
    \end{pmatrix},\hspace{3mm}
    \sigma_2= \begin{pmatrix}
    0&-i\\
    i&0
    \end{pmatrix},\hspace{3mm}
    \sigma_3 =\begin{pmatrix}
    1&0\\
    0&-1
    \end{pmatrix}.\nonumber
\end{equation}
We see that those matrices are unitary, self-adjoint and satisfy $\mathrm{Tr}(\sigma_i\sigma_j)=2\delta_{ij}$, where $\delta_{ij}$ is a Kronecker delta symbol (equal to one if $i=j$ and zero otherwise). In order to make the connection with the Bell basis, introduced in Section \ref{uvod}, we take $\beta_1=\sigma_0$, $\beta_2=\sigma_1$, $\beta_3=\sigma_3$ and $\beta_4=-i\sigma_2$. This implies that we have $\mathrm{Tr}(\beta_i\beta_j^\dagger)=2\delta_{ij}$, with the usual definition of matrix adjoint.

However, in order to check whether two diagrams commute, it is usually straightforward to include scalars into consideration. One can then just neglect this issue of scalars and work without explicitly using them (as done previously). They are, of course, needed if one is to obtain probabilities for different outcomes of a measurement, but in this work (and related work of \cite{AC04,CO05}) this is not a primary task.

\newpage

\begin{center}\textmd{\textbf{Acknowledgements} }
\end{center}
\smallskip
Zoran Petri\' c and Mladen Zeki\' c were supported by the Science Fund of the Republic of Serbia,
Grant No. 7749891, Graphical Languages - GWORDS. Du\v san \D{D}or\d{d}evi\' c was supported by the Faculty of Physics, University of Belgrade, through the grant of the Ministry of Education, Science, and Technological Development of the Republic of Serbia (Contract No. 451-03-68/2022-14/200162).

\begin{center}\textmd{\textbf{Data availability statement} }
\end{center}
\smallskip

No datasets were generated or analysed during the current study.

\end{document}